\documentclass[reqno,10pt,centertags,draft]{amsart}
\usepackage{amsmath,amsthm,amscd,amssymb,latexsym,upref,stmaryrd}

\newcommand{\bbN}{{\mathbb{N}}}
\newcommand{\bbR}{{\mathbb{R}}}

\newcommand{\bbZ}{{\mathbb{Z}}}
\newcommand{\bbC}{{\mathbb{C}}}

\newcommand{\cB}{{\mathcal B}}

\newcommand{\cE}{{\mathcal E}}

\newcommand{\cH}{{\mathcal H}}

\newcommand{\cM}{{\mathcal M}}

\newcommand{\cR}{{\mathcal R}}

\newcommand{\gC}{{\mathfrak{C}}}

\newcommand{\no}{\notag}
\newcommand{\lb}{\label}
\newcommand{\bi}{\bibitem}

\newcommand{\ol}{\overline}

\newcommand{\wti}{\widetilde}

\newcommand{\Oh}{O}
\newcommand{\oh}{o}
\newcommand{\f}{\frac}

\renewcommand{\Re}{\mathop\mathrm{Re}}
\renewcommand{\Im}{\mathop\mathrm{Im}}

\newcommand{\dom}{\text{\rm{dom}}}

\newcommand{\ran}{\text{\rm{ran}}}

\DeclareMathOperator{\tr}{tr}

\allowdisplaybreaks \numberwithin{equation}{section}

\newtheorem{theorem}{Theorem}[section]

\newtheorem{lemma}[theorem]{Lemma}

\newtheorem{hypothesis}[theorem]{Hypothesis}

\theoremstyle{remark}
\newtheorem{remark}[theorem]{Remark}

\begin{document}

\title[The Damped String Problem]{The Damped String Problem Revisited}

\author[F.\ Gesztesy]{Fritz Gesztesy}
\address{Department of Mathematics,
University of Missouri, Columbia, MO 65211, USA}
\email{gesztesyf@missouri.edu}
\urladdr{http://www.math.missouri.edu/personnel/faculty/gesztesyf.html}

\author[H.\ Holden]{Helge Holden}
\address{Department of Mathematical Sciences,
Norwegian University of
Science and Technology, NO--7491 Trondheim, Norway}
\email{holden@math.ntnu.no}
\urladdr{http://www.math.ntnu.no/~holden/}
\thanks{Supported in part by the Research Council of Norway.}
\date{\today}

\begin{abstract}
We revisit the damped string equation on a compact interval with a variety of boundary conditions and derive an infinite sequence of trace formulas associated with it, employing methods familiar from supersymmetric quantum mechanics. 
We also derive completeness and Riesz basis results (with parentheses) for the associated root functions under less smoothness assumptions on the coefficients than usual, using operator theoretic methods (rather than detailed eigenvalue and root function asymptotics) only.  
\end{abstract}

\maketitle


\section{Introduction}
\label{s1}

We reconsider the damped one-dimensional wave equation
\begin{align}
& \rho(x)^2 u_{tt}(x,t) - u_{xx}(x,t) + \alpha(x) u_t(x,t) = 0, \quad 
(x,t)\in (0,1) \times [0,\infty),    \no  \\
& u(\cdot,0) = f_0, \; u_t (\cdot,0) = f_1,    \lb{1.1} \\
& u(\cdot,t), \; t \in\bbR, \, \text{ satisfies certain boundary conditions at 
$x=0$ and $x=1$,}   \no
\end{align}
assuming
\begin{equation}
0 < \rho, \rho^{-1} \in L^{\infty} ([0,1]; dx), \quad \alpha \in L^{\infty} ([0,1]; dx), \; 
\text{$\alpha$ real-valued},   \lb{1.2}
\end{equation}
and appropriately choosing $f_j \in L^2 ([0,1]; \rho^2 dx)$, $j=0,1$. Suitable 
boundary conditions for $u(\cdot,t)$, $t \geq 0$, in \eqref{1.1} at $x=0$ and 
$x=1$ studied in this paper are, for instance, Dirichlet, Neumann, (anti)periodic boundary conditions, etc. (we refer to Section \ref{s3} for details).  

We note that the seemingly more general equation with variable 
speed $c(\cdot)>0$, 
\begin{equation}
\rho(x)^2 u_{tt}(x,t) - c(x)^2 u_{xx}(x,t) + \alpha(x) u_t(x,t) = 0, 
\quad (x,t)\in (0,1) \times [0,\infty),  \lb{1.3}
\end{equation}
subordinates to the case described in \eqref{1.1} as long as 
\begin{equation}
0 < c, c^{-1} \in L^{\infty} ([0,1]; dx),    \lb{1.4}
\end{equation}
replacing $\rho$ by $\rho/c^2$ and $\alpha$ by $\alpha/c^2$ in \eqref{1.1}.

Our interest in this topic and the principal motivation for writing this paper 
has its origin in a question posed by Steve Cox in September of 2008: He 
had derived the following trace formula 
(in the case $\rho =1$) 
\begin{equation}
{\tr}_{L^2([0,1]; dx)}\big(\alpha ((- d^2/dx^2)_D)^{-1}\big) 
= \int_0^1 dx \, x(1-x) \alpha(x),    \lb{1.5} 
\end{equation}
in the case where $u(\cdot,t)$, $t \geq 0$, in \eqref{1.1} satisfies Dirichlet 
boundary conditions at $x=0,1$ (in particular, the operator $(- d^2/dx^2)_D$ 
in \eqref{1.5} denotes the Dirichlet Laplacian in $L^2([0,1]; dx)$). Steve Cox 
then posed the question whether there actually 
exists an infinite sequence of such trace formulas in analogy to the well-known sequences of trace formulas for completely integrable evolution equations of soliton-type (e.g., the Korteweg--de Vries (KdV) equation). This question will 
indeed be answered affirmatively in Theorem \ref{t5.4} in the general case 
where $\rho$ is nonconstant and for a variety of boundary conditions at 
$x=0,1$. 

We note that the area of damped wave equations remains incredibly  
active up to this day. Since we cannot possibly describe the recent developments 
in detail in this paper, we refer, for instance, to \cite{AL03}, \cite{BRT82}, 
\cite{BI88}, \cite{BIW95}, \cite{BR00}, \cite{BF09}, \cite{CR82}, \cite{CT89}, 
\cite{CFNQ90}, \cite{CE10}, \cite{CH08}, \cite{CK96}, \cite{CO96}, 
\cite{CZ94}, \cite{CZ95}, \cite{En92}, \cite{En94}, \cite{Fa90}, \cite{Fr96}, 
\cite{Fr99}, \cite{FZ96}, \cite{GGHT10}, \cite{GS03}, \cite{Hi03}, \cite{HS99}, 
\cite{HS04}, \cite{Hu88}, \cite{Hu97}, 
\cite{JT07}, \cite{JT09}, \cite{JTW08}, \cite{Lo97}, \cite{PBR04}, \cite{Ph09}, 
\cite{Pi97}, \cite{Pi98}, \cite{Pi99}, 
\cite{RTY09}, \cite{Sh96a}, \cite{Sh97}, \cite{Sh97a}, \cite{Sh99}, \cite{Sh01}, 
\cite{SMDB97}, \cite{Sj00}, \cite{Tr09}, \cite{Ve04}, and the references therein, which lead the 
interested reader into a variety of directions.   

The traditional semigroup approach to the damped (abstract) wave 
equations \eqref{1.1} (cf., e.g., \cite[Sect.\ VI.3]{EN00}, \cite[Ch.\ VIII]{Fa85}) consists of rewriting it as a first-order system of the type 
\begin{align}
\begin{split}
& \begin{pmatrix} u(\cdot,t) \\[1mm] u_t(\cdot,t) \end{pmatrix}_t = 
\begin{pmatrix} 0 & I_{L^2([0,1]; \rho^2 dx)}  \\[1mm] 
- \rho^{-2} (-d^2/dx^2)_{\rm bc} & - \alpha \rho^{-2} 
\end{pmatrix}
\begin{pmatrix} u(\cdot,t) \\[1mm] u_t(\cdot,t) \end{pmatrix},     \lb{1.6} \\[1mm] 
& \begin{pmatrix} u(\cdot,0) \\ u_t (\cdot,0) \end{pmatrix} 
= \begin{pmatrix} f_0 \\ f_1 \end{pmatrix}, \quad t \geq 0,   
\end{split}
\end{align} 
where $\rho^{-2} (-d^2/dx^2)_{bc} \geq 0$ in $L^2([0,1]; \rho^2 dx)$ is of the 
type
\begin{equation}
\rho^{-2} (-d^2/dx^2)_{\rm bc} = T_{\rm bc}^*  T_{\rm bc}^{}     \lb{1.7}
\end{equation} 
with 
\begin{equation}
T_{\rm bc}^{}  = i \rho^{-1} (d/dx)_{bc}    \lb{1.8}
\end{equation} 
in $L^2([0,1]; \rho^2 dx)$, and the subscript ``bc'' represents appropriate 
boundary conditions at $x=0,1$ (Dirichlet, Neumann, (anti)periodic, etc.) 
to be detailed in Section \ref{s3}.

However, we will not be working with the generator 
\begin{align}
\begin{split} 
& i G_{T_{\rm bc}^{} ,\alpha/\rho^2} = \begin{pmatrix} 0 & I_{L^2([0,1]; \rho^2 dx)}  
\\[1mm]  
- \rho^{-2} (-d^2/dx^2)_{\rm bc} & - \alpha \rho^{-2} \end{pmatrix},    \lb{1.9} \\
& \dom(G_{T_{\rm bc}^{} ,\alpha/\rho^2}) = \dom(T_{\rm bc}^*  T_{\rm bc}^{}) \oplus 
\dom(T_{\rm bc}^{} ), 
\end{split} 
\end{align}
in the Hilbert space $\cH_{T_{\rm bc}^{} } \oplus L^2([0,1]; \rho^2 dx)$ associated with \eqref{1.6}, where 
\begin{align}
\begin{split} 
& \cH_{T_{\rm bc}^{} } = (\dom(T_{\rm bc}^{} ), (\cdot,\cdot)_{\cH_{T_{\rm bc}^{} }}),   \lb{1.10} \\
& (u,v)_{\cH_{T_{\rm bc}^{} }} = (T_{\rm bc}^{}  u, T_{\rm bc}^{}  v)_{L^2([0,1]; \rho^2 dx)},  \; 
u, v \in \dom(T_{\rm bc}^{} ).
\end{split}
\end{align} 
Instead, we will put the principal focus on the Dirac-type operator 
\begin{align}
\begin{split}
& D_{\rm bc} + B = \begin{pmatrix} 0 & T_{\rm bc}^*  \\[1mm] T_{\rm bc}^{}  & 0 \end{pmatrix} 
+ \f{i \alpha}{\rho^2} \begin{pmatrix} I_{L^2([0,1]; \rho^2 dx)} & 0 \\[1mm] 0 & 0 
\end{pmatrix},   \lb{1.11} \\[1mm] 
&\dom(D_{\rm bc} + B) = \dom(D_{\rm bc}) = \dom(T_{\rm bc}^{} ) \oplus \dom(T_{\rm bc}^* )    
\end{split}
\end{align}
in the Hilbert space $L^2([0,1]; \rho^2 dx)^2$ by employing the fact that 
$ i G_{T_{\rm bc}^{} ,\alpha/\rho^2}$ in $\cH_{T_{\rm bc}^{} } \oplus L^2([0,1]; \rho^2 dx)$ 
is unitarily equivalent to 
$(D_{\rm bc} + B)(I_{L^2([0,1]; \rho^2 dx)} \oplus P_{\ran(T_{\rm bc}^{} )})$ in  
$L^2([0,1]; \rho^2 dx)^2$ (cf.\ Theorem \ref{t2.4} for details), as recently 
proven in \cite{GGHT10} in the context of abstract wave equations with 
a damping term. 

Working with $D_{\rm bc} + B$ rather than using $G_{T_{\rm bc}^{} ,\alpha/\rho^2}$ has 
two distinct advantages: First, the unperturbed Dirac operator 
\begin{equation}
D_{\rm bc} = \begin{pmatrix} 0 & T_{\rm bc}^*  \\[1mm] T_{\rm bc}^{}  & 0 \end{pmatrix} 
\lb{1.12}
\end{equation}
in $L^2([0,1]; \rho^2 dx)^2$ is self-adjoint and of a supersymmetric nature 
(cf.\ Appendix \ref{sA} for details) which permits its spectral analysis in terms 
of the operator $T_{\rm bc}^*  T_{\rm bc}^{}  \geq 0$ (resp., $T_{\rm bc}^{}  T_{\rm bc}^*  \geq 0$), 
and second, the non-self-adjoint term $B$ is represented in terms of a simple 
diagonal $2\times 2$ block operator in $L^2([0,1]; \rho^2 dx)^2$. 

Assuming \eqref{1.1} and choosing $\zeta\in\bbR\backslash\{0\}$ with 
$|\zeta|$ sufficiently small, such that $\zeta \in \rho(D_{\rm bc} + B)$, we prove in 
Theorem \ref{t5.4} that  
\begin{align}
& {\tr}_{L^2([0,1]; \rho^2 dx)^2}\big(\Im \big[(D_{\rm bc} + B - \zeta I_2)^{-1}\big]\big)  \no \\
& \quad = \Im\big[{\tr}_{L^2([0,1]; \rho^2 dx)}
\big(\big(2 \zeta + i \big(\alpha/\rho^2\big)\big) 
(T_{\rm bc}^*  T_{\rm bc}^{}  - \zeta^2 I - \zeta i (\alpha/\rho^2))^{-1}\big)\big]     \no \\
& \quad \; = \sum_{m=0}^\infty t_{{\rm bc}, 2m} \, \zeta^{2m},    \lb{1.13} 
\end{align}
where
\begin{align}
t_{{\rm bc},0} &= {\tr}_{L^2([0,1]; \rho^2 dx)}\big((\alpha/\rho^2) (T_{\rm bc}^*  T_{\rm bc}^{} )^{-1}\big),   \lb{1.14} \\
t_{{\rm bc},2} &= - {\tr}_{L^2([0,1]; \rho^2 dx)}\big(\big(\alpha/\rho^2\big)
(T_{\rm bc}^*  T_{\rm bc}^{} )^{-1} 
\big(\alpha/\rho^2\big) (T_{\rm bc}^*  T_{\rm bc}^{} )^{-1} 
\big(\alpha/\rho^2\big) (T_{\rm bc}^*  T_{\rm bc}^{} )^{-1}\big)   \no \\
& \quad + 3 {\tr}_{L^2([0,1]; \rho^2 dx)}\big(\big(\alpha/\rho^2\big)
(T_{\rm bc}^*  T_{\rm bc}^{} )^{-2}\big),    \lb{1.15}  \\
& \hspace*{-3mm} \text{ etc.,}    \no 
\end{align}
and explicit boundary condition (bc) dependent expressions for $t_{{\rm bc},0}$ 
are listed in \eqref{5.11a}--\eqref{5.11d}. The following infinite sequence of 
trace formulas, our principal new result, is proved in Theorem \ref{t5.8}:
\begin{equation}
\sum_{j\in J} \f{\Im\big(\lambda_j(D_{\rm bc} + B)^{m+1}\big)}
{|\lambda_j(D_{\rm bc} + B)|^{2(m+1)}} 
= \begin{cases} -  t_{{\rm bc},2 n}, & m=2n, \\ 0, & m=2n+1,  \end{cases}  
\quad n \in \bbN_0.   \lb{1.16}
\end{equation}
Explicitly, one obtains for $m=0,1$ in \eqref{1.16}, 
\begin{align}
& \sum_{j\in J} \f{\Im(\lambda_j(D_{\rm bc} + B))}{|\lambda_j(D_{\rm bc} + B)|^2} 
= - t_{{\rm bc},0} = 
- {\tr}_{L^2([0,1]; \rho^2 dx)}\big((\alpha/\rho^2) (T_{\rm bc}^*  T_{\rm bc}^{} )^{-1}\big),     
\lb{1.17} \\[1mm]
& \sum_{j\in J} \f{\Im(\lambda_j(D_{\rm bc} + B)) 
\Re(\lambda_j(D_{\rm bc} + B))}{|\lambda_j(D_{\rm bc} + B)|^4} = 0,     \lb{1.18} \\
& \quad  \text{ etc.}    \no 
\end{align}

In Section \ref{s2} we succinctly consider abstract damped wave equations and 
detail the intimate spectral connections between (abstract versions of) 
$G_{T_{\rm bc}^{} ,\alpha/\rho^2}$ and $D_{\rm bc} + B$, a topic discussed in depth 
in \cite{GGHT10}. Section \ref{s3} is devoted to the self-adjoint supersymmetric 
Dirac-type operator $D_{\rm bc} + B$ in the Hilbert space $L^2([0,1]; \rho^2 dx)^2$ 
and the associated operators $T_{\rm bc}^{} $. In particular, we describe in detail the boundary conditions chosen at $x=0,1$ and the Green's function corresponding 
to $T_{\rm bc}^*  T_{\rm bc}^{} $, that is, the integral kernel of the resolvent 
$(T_{\rm bc}^*  T_{\rm bc}^{}  - z I)^{-1}$. In Section \ref{s4} we derive 
an explicit formula for (an abstract version of) the resolvent of $D_{\rm bc} + B$,  
employing the well-known expression of the resolvent of the supersymmetric 
Dirc-type operator $D_{\rm bc}$. Section \ref{s5} focuses on the infinite sequence 
of trace formulas for the damped string equation (cf.\  \eqref{1.13}--\eqref{1.18}), 
using several well-known results in the spectral theory of non-self-adjoint 
operators associated with the names of Schur, Livsic, and Keldysh. In addition, 
based on abstract results on the existence of a Riesz basis with parentheses for 
a certain class of non-self-adjoint perturbations of a normal operator due to 
Katsnelson \cite{Ka67a}, \cite{Ka67b}, Markus \cite{Ma62}, and Markus and Matsaev \cite{MM81}, \cite{MM84}, we show that $D_{\rm bc} + B$ possesses a Riesz basis with parentheses in $L^2([0,1]; \rho^2 dx)^2$ (cf.\ 
Theorem \ref{t5.12}). A host of useful (spectral) properties of abstract self-adjoint supersymmetric Dirac-type operators is collected in Appendix \ref{sA}. 

Finally, we briefly summarize some of the notation used in this paper: Let $\cH$ be a separable complex Hilbert space, $(\cdot,\cdot)_{\cH}$ the scalar product in $\cH$ (linear in the second factor), and $I_{\cH}$ the identity operator in $\cH$.
Next, let $T$ be a linear operator mapping (a subspace of) a
Banach space into another, with $\dom(T)$, $\ran(T)$, and $\ker(T)$ denoting the
domain, range, and kernel (i.e., null space) of $T$. The closure of a closable 
operator $S$ is denoted by $\ol S$. The spectrum, essential spectrum, point spectrum, discrete spectrum, and resolvent set of a closed linear operator in $\cH$ will be denoted by 
$\sigma(\cdot)$, $\sigma_{\rm ess}(\cdot)$, $\sigma_{\rm p}(\cdot)$, $\sigma_{\rm d}(\cdot)$, and 
$\rho(\cdot)$, respectively. The
Banach spaces of bounded and compact linear operators in $\cH$ are
denoted by $\cB(\cH)$ and $\cB_\infty(\cH)$, respectively. Similarly,
the Schatten--von Neumann (trace) ideals will subsequently be denoted
by $\cB_p(\cH)$, $p\in (0,\infty)$. Analogous notation $\cB(\cH_1,\cH_2)$,
$\cB_\infty (\cH_1,\cH_2)$, etc., will be used for bounded, compact,
etc., operators between two Hilbert spaces $\cH_1$ and $\cH_2$. In
addition, ${\tr}_{\cH}(T)$ denotes the trace of a trace class operator
$T\in\cB_1(\cH)$ and ${\det}_{\cH,k}(I_{\cH}+S)$ represents the (modified)
Fredholm determinant associated with an operator $S\in\cB_k(\cH)$,
$k\in\bbN$ (for $k=1$ we omit the subscript $1$ in $\det_{\cH}(\cdot)$). 
Finally, $P_{\cM}$ denotes the orthogonal projection onto a closed, linear 
subspace $\cM$ of $\cH$, $\dot +$ denotes the direct (not necessarily 
orthogonal) sum in $\cH$, and $\oplus$ abbreviates the direct orthogonal 
sum in $\cH$.

\section{Abstract Damped Wave Equations}
\lb{s2}

We start with some abstract considerations modeling damped wave equations.

\begin{hypothesis} \lb{h2.1}
Let $\cH$ be a complex separable Hilbert space. \\
$(i)$ Assume that $A$ is a densely defined, closed operator in $\cH$ satisfying  
\begin{equation}
A^* A \geq \varepsilon I_{\cH}    \lb{2.1}
\end{equation}
for some $\varepsilon >0$. \\
$(ii)$ Suppose that $R\in\cB(\cH)$. 
\end{hypothesis}

We note that various extensions of the condition $R\in\cB(\cH)$ are possible, see for instance,  \cite{BI88}, \cite{BIW95}, \cite[Sect.\ VI.3]{EN00}. 

Since $\ker(A^* A) = \ker(A)$ if $A$ is densely defined and closed in $\cH$, assumption 
\eqref{2.1} implies that 
\begin{equation}
\ker (A) = \ker (A^* A) = \{0\}.    \lb{2.1a} 
\end{equation}
In addition, we recall that \eqref{2.1} implies that $\ran(A)$ is a closed linear 
subspace of $\cH$. 

Since we are interested in an abstract version of the damped wave equation of 
the form
\begin{equation}
u_{tt} + R u_t + A^*A u =0, \quad u(0) = f_0, \;  u_t(0) = f_1, \quad t \geq 0,   \lb{2.2}
\end{equation}
we rewrite it in the familiar first-order form
\begin{equation}
\begin{pmatrix} u \\ u_t \end{pmatrix}_t =\begin{pmatrix} 0 & I_{\cH} \\ - A^* A & -R \end{pmatrix}
\begin{pmatrix} u \\ u_t \end{pmatrix}, \quad 
\begin{pmatrix} u(0) \\ u_t(0) \end{pmatrix} = \begin{pmatrix} f_0 \\ f_1 \end{pmatrix}, 
\quad t \geq 0,    \lb{2.3}
\end{equation}
and set up the abstract initial value problem \eqref{2.2} as follows: First, one introduces the Hilbert space 
$\cH_A$ by
\begin{equation}
\cH_A = (\dom(A); (\cdot,\cdot)_{\cH_A}), \quad (u,v)_{\cH_A} = (Au,Av)_{\cH}, \; u,v \in \dom(A)  \lb{2.4}
\end{equation}
($\cH_A$ is complete since by hypothesis, $A^* A \geq \varepsilon I_{\cH}$). Moreover, since  
$\ker(A) = \{0\}$, the polar decomposition of $A$ in $\cH$ is of the form $A = J_A |A|$, where $J_A$ 
is an isometry in $\cH$ and $|A| = (A^* A)^{1/2} \geq \varepsilon^{1/2} I_{\cH}$. Thus, 
$\|Af\|_{\cH} = \| |A| \|_{\cH}$, $f\in\dom(A)$, and hence one actually has 
\begin{equation}
\cH_A = \cH_{|A|} \, \text{ with } \, |A| \geq \varepsilon^{1/2} I_{\cH}.    \lb{2.4a}
\end{equation}
We emphasize, that while we assumed that $\ker(A) = \{0\}$, it may happen that $\ker(A^*) \supsetneqq \{0\}$ and hence  $\ran (A) \subsetneqq \cH$, since $\cH = \ker(A^*) \oplus \ol{\ran(A)}$. In particular, we are not assuming $0 \in \rho(A)$, but \eqref{2.1} implies $0\in \rho(|A|)$. 

Given $\cH_A$, one then studies the abstract Cauchy problem in the Hilbert space $\cH_A \oplus \cH$, 
\begin{equation}
F' (t) = G_{A,R} F(t),  \quad F(t) \in \dom (G_{A,R}), \; t \geq 0,  \quad 
F(0) = \begin{pmatrix} f_0 \\ f_1 \end{pmatrix},   \lb{2.5}
\end{equation}
where 
\begin{align}
& G_{A,R} = \begin{pmatrix} 0 & I_{\cH} \\ - A^* A & -R \end{pmatrix}, \quad 
\dom(G_{A,R}) = \dom(A^* A) \oplus \dom(A) \subseteq \cH_A \oplus \cH,    
\lb{2.6} \\
& F(t) = \begin{pmatrix} f(t) \\ g(t) \end{pmatrix} \in \dom (G_{A,R}), 
\quad t \geq 0,   \lb{2.7} 
\end{align}
and the scalar product in $\cH_A \oplus \cH$ is of course defined as usual by
\begin{align}
& \left(\begin{pmatrix} u_1 \\ v_1 \end{pmatrix}, 
\begin{pmatrix} u_2 \\ v_2 \end{pmatrix} \right)_{\cH_A \oplus \cH} 
= (u_1, u_2)_{\cH_A} + (v_1, v_2)_{\cH}  \no \\
& \quad \, = (A u_1, A u_2)_{\cH} + (v_1, v_2)_{\cH}, \quad 
\begin{pmatrix} u_1 \\ v_1 \end{pmatrix}, 
\begin{pmatrix} u_2 \\ v_2 \end{pmatrix} \in \cH_A \oplus \cH.   \lb{2.7a}
\end{align}
Since $i G_{A,0}$ is well-known to be self-adjoint in $\cH_A \oplus \cH$, generating the corresponding unitary group in $\cH_A \oplus \cH$,
\begin{equation}
e^{G_{A,0} t} =\begin{pmatrix} \cos\big((A^* A)^{1/2} t\big) & 
(A^* A)^{-1/2}  \sin\big((A^* A)^{1/2} t\big) \\[2mm]  - (A^* A)^{1/2}  \sin\big((A^* A)^{1/2} t\big) 
& \cos\big((A^* A)^{1/2} t\big) \end{pmatrix}, \quad t \geq 0    \lb{2.8}
\end{equation}
(cf.\ e.g., \cite[Sect.\ VI.3]{EN00}, \cite[Sect.\ 2.7]{Go85}, \cite{GW03}), and 
\begin{equation}
B_R = \begin{pmatrix} 0 & 0 \\ 0 & - R \end{pmatrix} \in \cB(\cH_A \oplus \cH),   \lb{2.9}
\end{equation}
$G_{A,R} = G_{A,0} + B_R$ generates a $C_0$ group on $\cH_A \oplus \cH$ (see, e.g., 
\cite[Sect.\ VI.3]{EN00}, \cite[Sect.\ 2.7]{Go85}, \cite{GW03}). 

Next, we make a connection between the point spectral properties of $i G_{A,R}$ in 
$\cH_A \oplus \cH$ and an abstract perturbed Dirac-type operator $Q + S$ in $\cH \oplus \cH$ defined by
\begin{align}
& Q = \begin{pmatrix} 0 & A^* \\ A & 0 \end{pmatrix}, \quad \dom(Q) = \dom(A) \oplus \dom(A^*) 
\subseteq \cH \oplus \cH,   \lb{2.10} \\
& S = \begin{pmatrix} -i R & 0 \\ 0 & 0 \end{pmatrix}, \quad \dom(S) = \cH \oplus \cH,   \lb{2.11} \\
& Q + S = \begin{pmatrix} -i R & A^* \\ A & 0 \end{pmatrix}, \quad 
\dom(Q + S) = \dom(A) \oplus \dom(A^*) \subseteq \cH \oplus \cH.   \lb{2.12}
\end{align}
For more details on the self-adjoint operator $Q$ we refer to Appendix \ref{sA}. 

We start with the following elementary result:  
\begin{lemma} \lb{l2.2}
Assume Hypothesis \ref{h2.1}. Then 
\begin{align}
\begin{split}
& \begin{pmatrix} -i R & A^* \\ A & 0 \end{pmatrix} 
\begin{pmatrix} 0 & I_{\cH} \\  - i A & 0 \end{pmatrix} =
\begin{pmatrix} 0 & I_{\cH} \\  - i A & 0 \end{pmatrix} 
i \begin{pmatrix} 0 & I_{\cH} \\  - A^* A  & -R \end{pmatrix}  \\
& \, \quad \text{on } \, \dom(A^* A) \oplus \dom(A) \subseteq \cH_A \oplus \cH.   \lb{2.13}
\end{split} 
\end{align}
\end{lemma}
\begin{proof}
A direct calculation yields for $u \in \dom(A^* A)$ and $v \in \dom(A)$,
\begin{equation}
\begin{pmatrix} -i R & A^* \\ A & 0 \end{pmatrix}
\begin{pmatrix} 0 & I_{\cH}\\ - i A & 0 \end{pmatrix} \begin{pmatrix}u \\ v \end{pmatrix} = 
\begin{pmatrix} -i A^* A u - i R v \\ Av \end{pmatrix} 
\end{equation}
and
\begin{equation}
\begin{pmatrix} 0 & I_{\cH} \\ -i A & 0 \end{pmatrix}
i \begin{pmatrix} 0 & I_{\cH} \\ - A^* A & - R \end{pmatrix} 
\begin{pmatrix} u \\ v \end{pmatrix} = 
\begin{pmatrix} -i A^* A u - i R v \\ Av \end{pmatrix}. 
\end{equation}
\end{proof}

In this context one observes that the facts $A^* A = |A|^2$ and 
$\dom(A) = \dom(|A|)$ permit one to replace the Dirac-type operator $Q$ by 
\begin{equation}
Q_{|A|} = \begin{pmatrix} 0 & |A| \\ |A| & 0 \end{pmatrix}, \quad 
\dom(Q_{|A|}) = \dom(A) \oplus \dom(A) 
\subseteq \cH \oplus \cH,   \lb{2.14} 
\end{equation} 
which can be advantageous as $|A| \geq \varepsilon^{1/2} I_{\cH}$ (cf.\ \eqref{2.4a}).

Next, we recall a few facts regarding eigenvalues of a densely defined, closed, linear operator  $T$ in $\cH$: The {\it geometric multiplicity}, $m_g(\lambda_0,T)$, of an eigenvalue $\lambda_0 \in \sigma_{\rm p}(T)$ of $T$ is given by 
\begin{equation}
m_g(\lambda_0,T) = \dim(\ker(T - \lambda_0 I_{\cH})),    \lb{2.15} 
\end{equation}
with $\ker(T - \lambda_0 I_{\cH})$ a closed linear subspace in $\cH$. 
(Here, and in the remainder of this paper, dimension always refers to the 
cardinality of an orthonormal basis (i.e., a complete orthonormal sequence of elements) in the separable pre-Hilbert space in question, cf.\ 
\cite[Sect.\ 3.3]{We80}.)

The set of all {\it root vectors} of $T$ (i.e., eigenvectors and generalized eigenvectors, or associated eigenvectors) corresponding to 
$\lambda_0 \in \sigma_{\rm p}(T)$ is given by 
\begin{equation} 
\cR(\lambda_0,T) = \big\{f\in\cH\,\big|\, (T - \lambda_0 I_{\cH})^k f = 0 \; 
\text{for some $k\in\bbN$}\big\}.     \lb{2.15a}
\end{equation}
The set $\cR(\lambda_0,T)$ is a linear subspace of $\cH$ whose dimension 
equals the {\it algebraic multiplicity}, $m_a(\lambda_0,T)$, of $\lambda_0$, 
\begin{equation}
m_a(\lambda_0,T) = \dim\big(\big\{f\in\cH\,\big|\, (T - \lambda_0 I_{\cH})^k f = 0 
\; \text{for some $k\in\bbN$}\big\}\big).   \lb{2.15b}
\end{equation}
In general, $\cR(\lambda_0,T)$ is not a closed linear subspace of $\cH$. 
($\cR(\lambda_0,T)$ is of course closed if $m_a(\lambda_0,T) < \infty$.) One has  
\begin{equation} 
m_g(\lambda_0,T) \leq m_a(\lambda_0,T).    \lb{2.15c}
\end{equation}

If in addition, the eigenvalue $\lambda_0$ of $T$ is an isolated point in $\sigma(T)$ (i.e., separated from the remainder of the spectrum of $T$), one can introduce 
the Riesz projection, $P(\lambda_0,T)$ of $T$ corresponding to $\lambda_0$, by 
\begin{equation}
P(\lambda_0,T)=-\f{1}{2\pi i}\oint_{C(\lambda_0; \varepsilon) }
d\zeta \, (T-\zeta I_{\cH})^{-1}, \lb{2.15d}
\end{equation}
with $C(\lambda_0; \varepsilon) $ a counterclockwise oriented circle centered at 
$\lambda_0$ with sufficiently small radius $\varepsilon>0$, such that the closed 
disk with center $\lambda_0$ and radius $\varepsilon$ excludes   
$\sigma(T)\backslash\{\lambda_0\}$. In this case,
\begin{equation}
\cH = \ker(P(\lambda_0,T)) \, \dot + \, \ran(P(\lambda_0,T)),    \lb{2.15e}
\end{equation} 
with $\ker(P(\lambda_0,T))$ and $\ran(P(\lambda_0,T))$ closed linear 
subspaces in $\cH$, and (cf.\ \cite[No.\ 149]{RS90})
\begin{equation}
\ran(P(\lambda_0,T)) = \big\{f \in \cH \,\big|\, 
\lim_{n\to\infty} \|(T - \lambda_0 I_{\cH})^n f\|^{1/n} = 0\big\},    \lb{2.15f}
\end{equation}
and hence,
\begin{equation}
\cR(\lambda_0,T) \subseteq \ran(P(\lambda_0,T)).    \lb{2.15g}
\end{equation}
Moreover, in the particular case where 
\begin{equation}
m_a(\lambda_0,T) < \infty, \, \text{ or equivalently, } \, 
\dim(\ran(P(\lambda_0,T)) < \infty, 
\end{equation} 
one has (cf.\ \cite[Sect.\ II.1]{GGK90}, \cite[Sects.\ I.1, I.2]{GK69})
\begin{equation}
m_a(\lambda_0,T) = \dim(\ran(P(\lambda_0,T)) < \infty, \quad 
\cR(\lambda_0,T) = \ran(P(\lambda_0,T)).      \lb{2.15h}
\end{equation}
If $T\in \cB_{\infty}(\cH)$, then any 
$\lambda_0 \in \sigma_{\rm p}(T)\backslash\{0\}$ satisfies \eqref{2.15h}. 

We also note for later purpose that if 
$\Lambda = \{\lambda_1,\dots,\lambda_N\}$ 
for some $N\in\bbN$, represents a finite cluster of 
eigenvalues of $T$, isolated from the remaining spectrum of $T$, then the 
Riesz projection of $T$ corresponding to $\Lambda$ is given by 
\begin{equation}
P(\Lambda,T) = - \f{1}{2\pi i}\oint_{C(\Lambda) }
d\zeta \, (T-\zeta I_{\cH})^{-1}, \lb{2.15ha}
\end{equation}
with $C(\Lambda)$ a counterclockwise oriented simple contour enclosing all eigenvalues in the cluster $\Lambda$ in its open interior, such that 
$C(\Lambda)$ excludes $\sigma(T)\backslash\Lambda$. In this situation 
$P(\Lambda,T)$ represents the sum of the disjoint Riesz projections 
corresponding to $\lambda_j \in \Lambda$, $j = 1,\dots,N$,  
\begin{equation}
P(\Lambda,T) = \sum_{j=1}^N P(\lambda_j,T), \quad 
P(\lambda_j,T) P(\lambda_k,T) = \delta_{j,k} P(\lambda_k,T), \; 
j,k = 1,\dots,N. 
\end{equation}

The connection between the point spectra of $i G_{A,R}$ and $(Q+S)$, more precisely, between $i G_{A,R}$ and 
$(Q+S)(I_{\cH} \oplus P_{\ran(A)})$, is detailed in the following result.  
In this context one observes that Hypothesis \ref{h2.1} implies that 
$\ran(A)$ is a closed linear subspace of $\cH$, 
\begin{equation}
\ran(A) = \ol{\ran(A)} = (\ker(A^*))^\bot.   \lb{2.15i}
\end{equation}

\begin{theorem} \lb{t2.3}
Assume Hypothesis \ref{h2.1}. Then 
\begin{equation}
0 \notin \sigma_{\rm p}(i G_{A,R}),    \lb{2.16}
\end{equation}
and 
\begin{equation}
\sigma_{\rm p}(i G_{A,R}) = \sigma_{\rm p}(Q + S)\backslash \{0\},   \lb{2.17}
\end{equation}
with geometric and algebraic multiplicities preserved. 

More precisely, let $0 \neq \lambda_0 \in \sigma_{\rm p} (i G_{A,R})$ and 
\begin{equation}
i G_{A,R} \begin{pmatrix} u(\lambda_0) \\ v(\lambda_0) \end{pmatrix} = 
\lambda_0 \begin{pmatrix} u(\lambda_0) \\ v(\lambda_0) \end{pmatrix}, \quad 
\begin{pmatrix} u(\lambda_0) \\ v(\lambda_0) \end{pmatrix} \in \dom (G_{A,R}),   \lb{2.20}
\end{equation} 
then $\lambda_0 \in \sigma_{\rm p} (Q+S)$ and 
\begin{equation}
(Q+S) \begin{pmatrix} v(\lambda_0) \\ -i A u(\lambda_0) \end{pmatrix} = 
\lambda_0 \begin{pmatrix} v(\lambda_0) \\ -i A u(\lambda_0) \end{pmatrix}, \quad 
\begin{pmatrix} v(\lambda_0) \\ -i A u(\lambda_0) \end{pmatrix} \in \dom (Q+S).  \lb{2.21}
\end{equation}
Conversely, let $0 \neq \lambda_1 \in \sigma_{\rm p} (Q+S)$ and 
\begin{equation}
(Q+S) \begin{pmatrix} \psi_1 (\lambda_1) \\ \psi_2 (\lambda_1) \end{pmatrix} = 
\lambda_1 \begin{pmatrix} \psi_1 (\lambda_1) \\ \psi_2 (\lambda_1) \end{pmatrix}, \quad 
\begin{pmatrix} \psi_1 (\lambda_1) \\ \psi_2 (\lambda_1) \end{pmatrix} \in \dom (Q+S),   \lb{2.22}
\end{equation}
then $\lambda_1 \in \sigma_{\rm p} (i G_{A,R})$, $\psi_2 (\lambda_1) \in \ran(A)$, and 
\begin{align}
\begin{split} 
i G_{A,R} \begin{pmatrix} i A^{-1}|_{\ran(A)} \psi_2 (\lambda_1) \\ \psi_1 (\lambda_1) \end{pmatrix} = 
\lambda_1\begin{pmatrix} i A^{-1}|_{\ran(A)} \psi_2 (\lambda_1) \\ \psi_1 (\lambda_1) \end{pmatrix},&   \\
\begin{pmatrix} i A^{-1}|_{\ran(A)} \psi_2 (\lambda_1) \\ \psi_1 (\lambda_1) \end{pmatrix} \in \dom (G_{A,R}),&    
\lb{2.23}
\end{split} 
\end{align} 
where $A^{-1}|_{\ran(A)}$ denotes the inverse of the map $A:\dom(A)\to\ran(A)$. 

We note that $Q+S$ in \eqref{2.17}, \eqref{2.21}, and \eqref{2.22} can be replaced by $(Q+S)(I_{\cH} \oplus P_{\ran(A)})$.
\end{theorem}
\begin{proof}
Let $(u(\lambda_0) \; v(\lambda_0))^\top \in \dom(G_{A,R})$ and suppose that 
$i G_{A,R} (u(\lambda_0) \; v(\lambda_0))^\top =0$, that is,
\begin{equation}
\begin{pmatrix} 0 & I_{\cH} \\ - A^* A & - R \end{pmatrix} 
\begin{pmatrix} u(\lambda_0) \\ v(\lambda_0)  \end{pmatrix} = 
\begin{pmatrix} v(\lambda_0) \\ 
- A^* A u(\lambda_0) - R v(\lambda_0) \end{pmatrix} =0,    \lb{2.24}
\end{equation}
implying $u(\lambda_0) = v(\lambda_0) =0$, and hence \eqref{2.16}. 

Next, let $0 \neq \lambda_0 \in \sigma_{\rm p} (i G_{A,R})$ and suppose \eqref{2.20} holds. Then, 
$(u(\lambda_0) \; v(\lambda_0))^\top \in \dom (G_{A,R})$ implies 
$u(\lambda_0) \in \dom (A^* A)$ and $v(\lambda_0) \in \dom(A)$ and 
using \eqref{2.13} one obtains   
\begin{align}
& \begin{pmatrix} - i R & A^* \\ A & 0 \end{pmatrix} 
\begin{pmatrix} v(\lambda_0) \\ -i A u(\lambda_0) \end{pmatrix} = 
\begin{pmatrix} - i R & A^* \\ A & 0 \end{pmatrix} 
\begin{pmatrix} 0 & I_{\cH} \\ A & 0 \end{pmatrix} 
\begin{pmatrix} u(\lambda_0) \\ v(\lambda_0) \end{pmatrix}  \no \\
& \quad = \begin{pmatrix} - i R & A^* \\ A & 0 \end{pmatrix} 
i \begin{pmatrix} 0 & I_{\cH} \\ - A^* A & - R \end{pmatrix} 
\begin{pmatrix} u(\lambda_0) \\ v(\lambda_0) \end{pmatrix} = 
\lambda_0 \begin{pmatrix} v(\lambda_0) \\ -i A u(\lambda_0) \end{pmatrix},    
\lb{2.25}
\end{align}
proving \eqref{2.21}. It is also clear that if 
$(u_{j}(\lambda_0) \; v_{j}(\lambda_0))^\top$, $j=1,2$, are two linearly independent 
nonzero solutions of $i G_{A,R} (u(\lambda_0) \; v(\lambda_0))^\top 
= \lambda_0 (u(\lambda_0) \; v(\lambda_0))^\top$, then also   
$(v_j(\lambda_0) \; -i A u_j(\lambda_0))^\top$, $j=1,2$, of 
$(Q+S) (\psi_1(\lambda_0) \;  \psi_2(\lambda_0))^\top = 
\lambda_0 (\psi_1(\lambda_0) \; \psi_2(\lambda_0))^\top$ are linearly independent 
nonzero solutions since by hypothesis, 
$\ker (A) = \{0\}$, proving that geometric multiplicities are preserved. 

Finally, supose that $0 \neq \lambda_1 \in \sigma_{\rm p} (Q+S)$ and assume that \eqref{2.22} holds. Then one obtains 
\begin{align}
& - i R \psi_1 (\lambda_1) + A^* \psi_2 (\lambda_1) = \lambda_1 \psi_1 (\lambda_1), \quad 
\psi_1 (\lambda_1) \in \dom (A), \; \psi_2 (\lambda_1) \in \dom(A^*),    \lb{2.26} \\
& \; A \psi_1 (\lambda_1) = \lambda_1 \psi_2 (\lambda_1), \, \text{ implying } \, 
\psi_2 (\lambda_1) \in \ran (A) = \dom\big( A^{-1}|_{\ran(A)}\big).     \lb{2.27}
\end{align}
Thus, one computes 
\begin{align}
& i \begin{pmatrix} 0 & I_{\cH} \\ - A^* A & - R \end{pmatrix} 
\begin{pmatrix} i A^{-1}|_{\ran(A)} \psi_2 (\lambda_1) \\ 
\psi_1 (\lambda_1)\end{pmatrix}    \no \\ 
& \quad = 
\begin{pmatrix} 0 & i A^{-1}|_{\ran(A)} \\ I_{\cH} & 0 \end{pmatrix} 
\begin{pmatrix} - i R & A^* \\ A & 0 \end{pmatrix}
\begin{pmatrix} \psi_1 (\lambda_1) \\ \psi_2 (\lambda_1) \end{pmatrix}    \no \\
& \quad = \lambda_1 \begin{pmatrix} 0 & i A^{-1}|_{\ran(A)} \\ 
I_{\cH} & 0 \end{pmatrix}  
\begin{pmatrix} \psi_1 (\lambda_1) \\ \psi_2 (\lambda_1) \end{pmatrix} = 
\lambda_1 \begin{pmatrix} i A^{-1}|_{\ran(A)} \psi_2 (\lambda_1) \\ 
\psi_1 (\lambda_1)\end{pmatrix}.    \lb{2.28}
\end{align}
Again one infers that geometric multiplicities are preserved as the map 
$A^{-1}|_{\ran(A)}:\ran(A)\to\dom(A)$ is injective. 

Finally, the preservation of algebraic multiplicities follows from the 
unitary equivalence result in Theorem \ref{t2.4} below.
\end{proof}

In fact, one can prove the following extension of Theorem \ref{t2.3}, and 
we refer to \cite{GGHT10} for a detailed proof: 

\begin{theorem}  [\cite{GGHT10}] \lb{t2.4}
Assume Hypothesis \ref{h2.1}. Then 
$\cH \oplus (\ker(A^*))^\bot = \cH \oplus \ran(A)$ is a reducing subspace 
for $Q+S$ and  
\begin{align}
\begin{split}
& (Q + S) (I_{\cH} \oplus [I_{\cH} - P_{\ker(A^*)}]) 
= (Q + S) (I_{\cH} \oplus P_{\ran(A)})   \\
& \quad = U_{\wti A} \, i \, G_{A,R} U_{\wti A}^{-1},    \lb{2.36}
\end{split} 
\end{align}
where $\wti A$ defined by 
\begin{equation}
\wti A : \begin{cases} \cH_A \to \ran(A), \\  f \mapsto Af, \end{cases}   
\, \text{ is unitary,}      \lb{2.36a}
\end{equation}
and 
\begin{align}
U_{\wti A} &= \begin{pmatrix} 0 & I_{\cH} \\ -i \, \wti A & 0 \end{pmatrix} \in 
\cB(\cH_A \oplus \cH, \cH \oplus \ran(A))  \, \text{ is unitary,}    \lb{2.37} \\
U_{\wti A}^{-1} &= \begin{pmatrix} 0 & i {\wti A}^{-1} \\ I_{\cH}  & 0 \end{pmatrix} 
\in \cB(\cH \oplus \ran(A), \cH_A \oplus \cH)  \, \text{ is unitary.}     \lb{2.39}
\end{align}
\end{theorem}

We note that while Theorem \ref{t2.4} has been obtained by Huang 
\cite{Hu97} under the additional assumption $|A| \geq \varepsilon^{1/2} I_{\cH}$, 
the results in Theorems \ref{t2.3} and \ref{t2.4} appear to be new under 
the general Hypothesis \ref{h2.1}. 

Next, let $\gC$ be a conjugation operator in $\cH$, that is, $\gC $ is an antilinear 
involution satisfying (see, e.g., \cite[Sect.\ III.5]{EE89} and \cite[p.\ 76]{Gl65}) 
\begin{equation}
	(\gC u,v)_\cH=(\gC v,u)_\cH  \quad u,v\in \cH, \quad \gC^2=I_{\cH}.    \lb{2.46}
\end{equation}
In particular,
\begin{equation}
	(\gC u,\gC v)_\cH=(v,u)_\cH, \quad u,v\in \cH. \lb{2.47}
\end{equation}
The densely defined operator $S$ in $\cH$ is called $\gC$-invariant if  
\begin{equation}
	S=\gC \, S \, \gC.     \lb{2.48}
\end{equation}

\begin{lemma} \lb{l2.5}
Assume Hypothesis \ref{h2.1}, let $\gC$ be a conjugation operator in $\cH$ and suppose that $A^*A$ and $R$ are $\gC $-invariant. Then
\begin{equation}
\lambda_0 \in \sigma_{\rm p} (i G_{A,R}) \, \text{ if and only if } \, 
- \ol{\lambda_0} \in \sigma_{\rm p} (i G_{A,R}),      \lb{2.49}
\end{equation}
with geometric and algebraic multiplicities preserved. As a consequence of 
Theorem \ref{t2.4}, 
the analogous relation \eqref{2.49} $($including preservation of geometric and algebraic multiplicities$)$ extends to nonzero eigenvalues of $Q+S$.
\end{lemma}
\begin{proof}
This follows from
\begin{equation}
\begin{pmatrix} \gC & 0 \\ 0 & \gC \end{pmatrix} i G_{A,R} 
\begin{pmatrix} \gC & 0 \\ 0 & \gC \end{pmatrix} = - i G_{A,R}. 
\end{equation}
\end{proof}

The following remark illustrates the connection between the generator 
$i G_{A,R}$ and the monic second-degree operator polynomial 
$L(z) = z^2 I_{\cH} + z i R - A^*A$, $z\in\bbC$, in $\cH$:

\begin{remark} \lb{r2.5}
The eigenvalue problem for $i G_{A,R}$, that is,
\begin{equation}
i G_{A,R} \begin{pmatrix} u(\lambda_0) \\ v(\lambda_0) \end{pmatrix} = 
\lambda_0 \begin{pmatrix} u(\lambda_0) \\ v(\lambda_0) \end{pmatrix}, \quad 
\begin{pmatrix} u(\lambda_0) \\ v(\lambda_0) \end{pmatrix} \in \dom (G_{A,R}),   \lb{2.29}
\end{equation} 
is equivalent to the pair of equations
\begin{equation}
v(\lambda_0) = - i \lambda_0 u(\lambda_0), \quad 
A^* A u(\lambda_0) + R v(\lambda_0) - i \lambda_0 v(\lambda_0) = 0.    \lb{2.30}
\end{equation}
Thus, eliminating the component $v(\lambda_0)$ in the second equation in \eqref{2.30} 
yields the quadratic pencil equation for $u(\lambda_0)$, 
\begin{equation}
\big(A^* A - \lambda_0 i R -\lambda_0^2\big) u(\lambda_0) =0, \quad u(\lambda_0) \in \dom (A^* A). 
\lb{2.31}
\end{equation}
Similarly, eliminating $ \psi_2 (\lambda_1)$ in the eigenvalue equation \eqref{2.22} for $(Q+S)$ results in the analogous quadratic pencil equation for $\psi_1 (\lambda_1)$, 
\begin{equation}
\big(A^* A - \lambda_1 i R -\lambda_1^2\big) \psi_1 (\lambda_1) =0, \quad 
\psi_1 (\lambda_1) \in \dom (A^* A).    \lb{2.32}
\end{equation}
Moreover, introducing the quadratic pencil (i.e., the second-degree  operator polynomial) $L(\cdot)$ in $\cH$, 
\begin{equation}
L(z) = z^2 I_{\cH} + z i R - A^*A, \quad 
\dom(L(z)) = \dom(A^* A), \quad z \in \bbC,     \lb{2.33a}
\end{equation}
one verifies the identity 
\begin{equation}
(L(z) \oplus I_{\cH}) F(z) = 
E(z) (i G_{A,R} - z I_{\cH \oplus \cH}), \quad z \in \bbC,   \lb{2.33b}
\end{equation}
where (for $z \in \bbC$)
\begin{align}
& E(z) = \begin{pmatrix} - z I_{\cH} - i R & -i I_{\cH} \\
I_{\cH} & 0 \end{pmatrix},  \quad \dom(E(z)) = 
\dom(R) \oplus \cH \subseteq \cH \oplus \cH,   \\
& E(z)^{-1} = \begin{pmatrix} 0 & I_{\cH} \\
i I_{\cH} & -i(- z I_{\cH} - i R) \end{pmatrix},  \quad 
\dom\big(E(z)^{-1}\big) = 
\cH \oplus \dom(R) \subseteq \cH \oplus \cH,   \\
& F(z) = \begin{pmatrix} I_{\cH} & 0 \\
- z I_{\cH} & i I_{\cH} \end{pmatrix}, \, 
F(z)^{-1} = \begin{pmatrix} I_{\cH} & 0 \\
- i z I_{\cH} & -i I_{\cH} \end{pmatrix} \in \cB(\cH \oplus \cH). 
\end{align}
Thus, identity \eqref{2.33b} exhibits $i G_{A,R}$ as a global linearization 
of the quadratic pencil $L(\cdot)$ (in analogy to the discussion in 
\cite[Sect.\ 1.1, Example\ 1.1.4]{Ro89} in the context of bounded 
operator pencils), and again the (geometric and algebraic) multiplicity of 
nonzero eigenvalues of $L(\cdot)$ and $i G_{A,R}$ coincide by definition. 
Of course, the unitary equivalence described in Theorem \ref{t2.4} also 
exhibits $(Q + S) (I_{\cH} \oplus P_{\ran(A)})$ as a global linearization 
of $L(\cdot)$. 

We note that even though the pencil $L(\cdot)$ has unbounded coefficients,
replacing $L(\cdot)$ by $L(\cdot)(A^*A + I_{\cH})^{-1}$ reduces matters to a 
pencil with bounded coefficients, in particular,  
\cite[Lemmas\ 20.1 and 20.2]{Ma88} apply to the spectrum of $L(\cdot)$ in 
this context.  

Finally, we note that the standard separation of variables argument, making the ansatz
\begin{equation}
u (t) = e ^{-i \lambda t} u(\lambda), \quad t \in \bbR, \; \lambda \in\bbC,   \lb{2.33}
\end{equation}
and inserting it into the equation 
\begin{equation}
\begin{pmatrix} u(t) \\ u_t (t) \end{pmatrix}_t = G_{A,R} \begin{pmatrix} u(t) \\ u_t (t) \end{pmatrix}, 
\quad t \geq 0,     \lb{2.34}
\end{equation}
 then yields of course the familiar eigenvalue problem
 \begin{equation}
i G_{A,R} \begin{pmatrix} u(\lambda) \\ - i \lambda u(\lambda) \end{pmatrix} 
= \lambda \begin{pmatrix} u(\lambda) \\ - i \lambda u(\lambda) \end{pmatrix},      
\lb{2.35}  
 \end{equation}
 compatible with \eqref{2.29} and \eqref{2.30}. 
 \end{remark}

\section{Supersymmetric Dirac-Type Operators}
\label{s3}

In this section we study self-adjoint supersymmetric Dirac-type operators $D$ 
of the type 
\begin{equation}
D = \begin{pmatrix} 0 & T^* \\ T & 0 \end{pmatrix}, \quad 
\dom(D) = \dom(T) \oplus \dom(T^*)   \lb{3.1}
\end{equation} 
in the Hilbert space 
\begin{equation} 
L^2([0,1]; \rho^2 dx)^2 = L^2([0,1]; \rho^2 dx) \oplus L^2([0,1]; \rho^2 dx) \simeq 
L^2([0,1]; \rho^2 dx) \otimes \bbC^2,    \lb{3.2}
\end{equation} 
where $T$ (and its adjoint $T^*$) are densely defined and closed first-order differential operators in 
$L^2([0,1]; \rho^2 dx)$ of the form
\begin{equation}
T = \f{i}{\rho} \f{d}{dx}, \quad T^* = \f{i}{\rho^2} \f{d}{dx} \rho,    \lb{3.3}
\end{equation}
with appropriate boundary conditions at $x=0,1$, and $\rho>0$ is an appropriate weight function, which, throughout this paper, will be assumed to 
satisfy the following conditions:

\begin{hypothesis}  \lb{h3.1}
Suppose that 
\begin{equation}
0 < \rho \in L^\infty([0,1]; dx), \quad 1/\rho \in L^\infty([0,1]; dx).    \lb{3.6}
\end{equation}
\end{hypothesis}

Assuming Hypothesis \ref{h3.1}, we now introduce the following concrete models for the 
operator $T$ in $L^2([0,1]; \rho^2 dx)$:
\begin{align}
& T_{\max} f = (i/\rho) f',     \lb{3.7} \\ 
& f\in \dom(T_{\max}) = \big\{g\in L^2([0,1]; \rho^2 dx) \, \big| \, g \in AC([0,1]);  \,       
g'  \in L^2([0,1]; \rho^2 dx)\big\},     \no \\
& T_{\min} f = (i/\rho) f',   \no \\
& f\in \dom(T_{\min}) = \big\{g\in L^2([0,1]; \rho^2 dx) \, \big| \, g \in AC([0,1]);  \, g(0)=0=g(1);     
\lb{3.8} \\
&  \hspace*{7.9cm} g'  \in L^2([0,1]; \rho^2 dx)\big\},   \no \\
& T_{0} f = (i/\rho) f',   \no \\
& f\in \dom(T_{0}) = \big\{g\in L^2([0,1]; \rho^2 dx) \, \big| \, g \in AC([0,1]);  \, g(0)=0;     
\lb{3.9} \\
&  \hspace*{6.45cm} g'  \in L^2([0,1]; \rho^2 dx)\big\},   \no \\
& T_{1} f = (i/\rho) f',   \no \\
& f\in \dom(T_{1}) = \big\{g\in L^2([0,1]; \rho^2 dx) \, \big| \, g \in AC([0,1]);  \, g(1)=0;     
\lb{3.10} \\
&  \hspace*{6.45cm} g'  \in L^2([0,1]; \rho^2 dx)\big\},   \no \\
& T_{\omega} f = (i/\rho) f',   \quad \omega \in\bbC\backslash\{0\},   \no \\
& f\in \dom(T_{\omega}) = \big\{g\in L^2([0,1]; \rho^2 dx) \, \big| \, g \in AC([0,1]); \, 
\, g(1)= \omega g(0);      \lb{3.11} \\
&  \hspace*{7.25cm}  g'  \in L^2([0,1]; \rho^2 dx)\big\},   \no
\end{align}
where $AC([0,1])$ denotes the set of absolutely continuous functions on $[0,1]$. 
Due to the conditions \eqref{3.6} on $\rho$, the fact that all operators in 
\eqref{3.7}--\eqref{3.10} are closed and densely defined in $L^2([0,1]; \rho^2 dx)$ 
parallels the well-known special case where $\rho =1$ a.e. 

The associated adjoint operators to \eqref{3.7}--\eqref{3.10} are then given by  
\begin{align}
& T_{\max}^*  f = (i/\rho^2) (\rho f)',   \no \\
& f\in \dom(T_{\max}^* ) = \big\{g\in L^2([0,1]; \rho^2 dx) \, \big| \, \rho g \in AC([0,1]);  
\, (\rho g)(0)=0=(\rho g)(1);     \no \\
&  \hspace*{7.3cm} (\rho g)'  \in L^2([0,1]; \rho^2 dx)\big\},   \lb{3.12} \\
& T_{\min}^*  f = (i/\rho^2) (\rho f)',     \lb{3.13} \\
& f\in \dom(T_{\min}^* ) = \big\{g\in L^2([0,1]; \rho^2 dx) \, \big| \, \rho g \in AC([0,1]); \, 
 (\rho g)' \in L^2([0,1]; \rho^2 dx)\big\},      \no \\
& T_0^* f = (i/\rho^2) (\rho f)',   \no \\
& f\in \dom(T_0^*) = \big\{g\in L^2([0,1]; \rho^2 dx) \, \big| \, \rho g \in AC([0,1]);  \, (\rho g)(1)=0;     
\lb{3.14} \\
&  \hspace*{6.7cm} (\rho g)' \in L^2([0,1]; \rho^2 dx)\big\},   \no \\
& T_1^* f = (i/\rho^2) (\rho f)',   \no \\
& f\in \dom(T_1^*) = \big\{g\in L^2([0,1]; \rho^2 dx) \, \big| \, \rho g \in AC([0,1]);  \, (\rho g)(0)=0;     
\lb{3.15} \\
&  \hspace*{6.7cm} (\rho g)' \in L^2([0,1]; \rho^2 dx)\big\},   \no \\
& T_{\omega}^*  f = (i/\rho^2) (\rho f)',   \quad \omega \in\bbC\backslash\{0\},   \no \\
& f\in \dom(T_{\omega}^* ) = \big\{g\in L^2([0,1]; \rho^2 dx) \, \big| \, \rho g \in AC([0,1]); \, 
\, (\rho g)(1)= (1/{\ol \omega}) (\rho g)(0);     \no \\
&  \hspace*{7.3cm}  (\rho g)' \in L^2([0,1]; \rho^2 dx)\big\}.    \lb{3.16}  
\end{align}

The sufficiency of the boundary conditions in \eqref{3.12}--\eqref{3.16} is clear from an elementary 
integration by parts
\begin{align}
& \int_0^1 \rho(x)^2  dx \, \ol{f(x)} [(1/\rho(x)) g'(x)] = \int_0^1 dx \, \ol{[\rho (x) f(x)]} g'(x)  \no \\
& \quad = \ol{[\rho (x) f(x)]} g(x)\big|_{0}^1 - \int_0^1 dx \, \ol{[\rho (x) f(x)]'} g(x)   \no \\
& \quad = \ol{[\rho (x) f(x)]} g(x)\big|_{0}^1 
- \int_0^1 \rho(x)^2 dx \, \ol{[1/\rho(x)^2] [\rho (x) f(x)]'} g(x).   \lb{3.17} 
\end{align}
We omit the details of the necessity of these boundary conditions which are well-known in the special case $\rho =1$ a.e. 

Replacing $T$ in \eqref{3.1} by $T_{\max}$, $T_{\min}$, $T_0$, $T_1$, and 
$T_{\omega}$, the corresponding self-adjoint Dirac-type operators in $L^2([0,1]; \rho^2 dx)^2$ are then denoted by $D_{\max}$, $D_{\min}$, $D_0$, $D_1$, 
and $D_{\omega}$, respectively.

Next we analyze the nullspaces of all these operators:

\begin{lemma}  \lb{l3.2}
Assume Hypothesis \ref{h3.1}. Then
\begin{align}
& \ker (T_{\max}) = \bbC, \quad \ker(T_{\max}^* ) = \{0\},     \lb{3.18} \\
& \ker (T_{\min}) = \{0\}, \quad \ker(T_{\min}^* ) = \{c/\rho \,|\, c \in\bbC\},     \lb{3.19} \\
& \ker (T_{0}) = \{0\}, \quad \ker(T_0^*) = \{0\},     \lb{3.20} \\
& \ker (T_{1}) = \{0\}, \quad \ker(T_1^*) = \{0\},     \lb{3.21} \\
& \ker (T_{\omega}) = \begin{cases} \{0\}, & \omega \in \bbC\backslash\{0,1\}, \\ 
\bbC, & \omega =1, \end{cases} 
\quad \ker(T_{\omega}^* ) = \begin{cases} \{0\}, & \omega \in \bbC\backslash\{0,1\}, \\
\{c/\rho \,|\, c\in\bbC\}, & \omega = 1, \end{cases}       \lb{3.22} 
\end{align} 
and hence\footnote{For simplicity, $0$ denotes the zero vector in 
$L^2([0,1]; \rho^2 dx)$ as well as in $L^2([0,1]; \rho^2 dx)^2$.} 
\begin{align}
& \ker (D_{\max}) = \bbC \oplus \{0\},     \lb{3.23} \\
& \ker (D_{\min}) = \{0\} \oplus \{c/\rho \,|\, c \in\bbC\},     \lb{3.24} \\
& \ker (D_{0}) = \{0\},     \lb{3.25} \\
& \ker (D_{1}) = \{0\},     \lb{3.26} \\
& \ker (D_{\omega}) = \begin{cases} \{0\}, & \omega \in \bbC\backslash\{0,1\}, \\ 
\bbC \oplus \{c/\rho \,|\, c\in\bbC\}, & \omega =1. \end{cases}    \lb{3.27}
\end{align}
\end{lemma}
\begin{proof}
Since 
\begin{equation}
\f{i}{\rho(x)} \bigg(\f{d}{dx} c\bigg) = 0, \quad 
\f{i}{\rho(x)^2} \bigg[\f{d}{dx}\bigg( \rho(x) \bigg(\f{c}{\rho(x)}\bigg)\bigg)\bigg] = 0, \quad c\in\bbC,    \lb{3.28}
\end{equation}
in the sense of distributions on $(0,1)$, and in each of these cases the operators $T$ 
and $T^*$ are of first order and hence have at most a one-dimensional nullspace, one 
concludes that $\bbC \subset L^2([0,1]; \rho^2 dx)$ (resp.,  
$\{c/\rho\,|\, c\in\bbC\} \subset L^2([0,1]; \rho^2 dx)$) are the kernels for $T$ (resp.,  
$T^*$) if and only if $c$ (resp., $c/\rho$) satisfy the boundary conditions 
in $T$ (resp., $T^*$). Checking whether the boundary conditions are fulfilled is 
elementary and yields \eqref{3.18}--\eqref{3.22}.

Relations \eqref{3.23}--\eqref{3.27} then follow from the general fact  
that $\ker(D) = \ker(T) \oplus \ker(T^*)$ (cf.\ \eqref{A.21}). 
\end{proof}

We continue by listing the corresponding operators $T^* T$:
\begin{align}
& T_{\max}^*  T_{\max} f = -(1/\rho^2) f'',   \no \\
& f \in \dom(T_{\max}^*  T_{\max}^{} ) =  \big\{g \in L^2([0,1]; \rho^2 dx) \,\big|\, g, g' \in AC([0,1]);  \lb{3.29} \\ 
& \hspace*{3.65cm}  g'(0) = 0 = g'(1); \, g'' \in L^2([0,1]; \rho^2 dx)\big\},  \no \\ 
& T_{\min}^*  T_{\min}^{}  f = -(1/\rho^2) f'',   \no \\
& f \in \dom(T_{\min}^*  T_{\min}^{} ) =  \big\{g \in L^2([0,1]; \rho^2 dx) \,\big|\, g, g' \in AC([0,1]);  \lb{3.30} \\ 
& \hspace*{3.7cm}  g(0) = 0 = g(1); \, g'' \in L^2([0,1]; \rho^2 dx)\big\},  \no \\ 
& T_0^* T_{0}^{} f = -(1/\rho^2) f'',   \no \\
& f \in \dom(T_0^* T_{0}^{}) =  \big\{g \in L^2([0,1]; \rho^2 dx) \,\big|\, g, g' \in AC([0,1]);  \lb{3.31} \\ 
& \hspace*{3.05cm}  g(0) = 0 = g'(1); \, g'' \in L^2([0,1]; \rho^2 dx)\big\},  \no \\ 
& T_1^* T_{1}^{} f = -(1/\rho^2) f'',   \no \\
& f \in \dom(T_1^* T_{1}^{}) =  \big\{g \in L^2([0,1]; \rho^2 dx) \,\big|\, g, g' \in AC([0,1]);  \lb{3.32} \\ 
& \hspace*{3.05cm}  g'(0) = 0 = g(1); \, g'' \in L^2([0,1]; \rho^2 dx)\big\},  \no \\ 
& T_{\omega}^*  T_{\omega}^{}  f = -(1/\rho^2) f'',   \no \\
& f \in \dom(T_{\omega}^*  T_{\omega}^{} ) =  \big\{g \in L^2([0,1]; \rho^2 dx) \,\big|\, g, g' \in AC([0,1]);  
\lb{3.33} \\ 
& \hspace*{3.05cm}  g(1) = \omega g(0), \, g'(1) = (1/{\ol \omega}) g'(0); \, 
g'' \in L^2([0,1]; \rho^2 dx)\big\}.  \no 
\end{align}
Analogously, one obtains for the operators $T T^*$: 
\begin{align}
& T_{\max}^{}  T_{\max}^*  f = -(1/\rho) [(1/\rho^2)(\rho f)']',   \no \\
& f \in \dom(T_{\max}^*  T_{\max}^{} ) =  \big\{g \in L^2([0,1]; \rho^2 dx) \,\big|\, \rho g, 
\rho^{-2} (\rho g)' \in AC([0,1]);  \lb{3.34} \\ 
& \hspace*{2.5cm}  (\rho g)(0) = 0 = (\rho g)(1); \, 
[(1/\rho^2)(\rho g)']' \in L^2([0,1]; \rho^2 dx)\big\},  \no \\  
& T_{\min}^{}  T_{\min}^*  f = -(1/\rho) [(1/\rho^2)(\rho f)']',   \no \\
& f \in \dom(T_{\min}^*  T_{\min}^{} ) =  \big\{g \in L^2([0,1]; \rho^2 dx) \,\big|\, \rho g, 
\rho^{-2} (\rho g)' \in AC([0,1]);  \lb{3.35} \\ 
& \hspace*{2.2cm}  (\rho g)'(0) = 0 = (\rho g)'(1); \, 
[(1/\rho^2)(\rho g)']' \in L^2([0,1]; \rho^2 dx)\big\},  \no \\ 
& T_{0} T_0^* f = -(1/\rho) [(1/\rho^2)(\rho f)']',   \no \\
& f \in \dom(T_0^* T_{0}^{}) =  \big\{g \in L^2([0,1]; \rho^2 dx) \,\big|\, \rho g, 
\rho^{-2} (\rho g)' \in AC([0,1]);  \lb{3.36} \\ 
& \hspace*{1.7cm}  (\rho g)'(0) = 0 = (\rho g)(1); \, 
[(1/\rho^2)(\rho g)']' \in L^2([0,1]; \rho^2 dx)\big\},  \no \\  
& T_{1}^{} T_1^* f = -(1/\rho) [(1/\rho^2)(\rho f)']',   \no \\
& f \in \dom(T_1^* T_{1}^{}) =  \big\{g \in L^2([0,1]; \rho^2 dx) \,\big|\, \rho g, 
\rho^{-2} (\rho g)' \in AC([0,1]);  \lb{3.37} \\ 
& \hspace*{1.7cm}  (\rho g)(0) = 0 = (\rho g)'(1); \, 
[(1/\rho^2)(\rho g)']' \in L^2([0,1]; \rho^2 dx)\big\},  \no \\  
& T_{\omega}^{}  T_{\omega}^*  f = -(1/\rho) [(1/\rho^2)(\rho f)']',   \no \\
& f \in \dom(T_{\omega}^*  T_{\omega}^{} ) =  \big\{g \in L^2([0,1]; \rho^2 dx) \,\big|\, \rho g, 
\rho^{-2} (\rho g)' \in AC([0,1]);  \lb{3.38} \\ 
& \hspace*{.4cm}  (\rho g)(1) = (1/{\ol \omega}) (\rho g)(0), \, (\rho g)'(1)= \omega (\rho g)'(1); \, 
[(1/\rho^2)(\rho g)']' \in L^2([0,1]; \rho^2 dx)\big\}.  \no 
\end{align}

One recalls that the spectra of all these operators $T^* T$ and $T T^*$ in 
\eqref{3.29}--\eqref{3.38} are purely discrete, that is, 
\begin{align}
& \sigma(T_{\max, \min, 0,1,\omega}^* T_{\max, \min, 0,1,\omega}^{}) 
= \sigma_{\rm d} (T_{\max, \min, 0,1,\omega}^* T_{\max, \min, 0,1,\omega}^{}),  \no \\ 
&\sigma_{\rm ess} (T_{\max, \min, 0,1,\omega}^* T_{\max, \min, 0,1,\omega}^{}) = \emptyset,    \lb{3.38A}  \\
& \sigma(T_{\max, \min, 0,1,\omega}^{} T_{\max, \min, 0,1,\omega}^*) 
= \sigma_{\rm d} (T_{\max, \min, 0,1,\omega}^{} T_{\max, \min, 0,1,\omega}^*),   \no \\ 
& \sigma_{\rm ess} (T_{\max, \min, 0,1,\omega}^{} T_{\max, \min, 0,1,\omega}^*) = \emptyset,   \lb{3.38a} 
\end{align}
in addition, we have of course (cf.\ \eqref{A.8} and \eqref{A.10})
\begin{equation}
\sigma(T_{\max, \min, 0,1,\omega}^* T_{\max, \min, 0,1,\omega}^{})\backslash\{0\} 
= \sigma(T_{\max, \min, 0,1,\omega}^{} T_{\max, \min, 0,1,\omega}^*)\backslash\{0\}     \lb{3.38b}
\end{equation}
and 
\begin{equation}
\ker(T^* T^{}) = \ker(T), \quad \ker(T^{} T^*) = \ker(T^*).    \lb{3.38ba}
\end{equation}
Here, and occasionally later on, $T$ stands for $T_{\max}$, $T_{\min}$, $T_{0}$, $T_{1}$, or $T_{\omega}$. 

For subsequent purpose we next describe the connection between 
$T_{\max, \min, 0,1,\omega}$ and $T_{\max, \min, 0,1,\omega}^*$ 
in the Hilbert space $L^2([0,1]; \rho^2 dx)$ and the corresponding operators 
\begin{align}
\begin{split}
\bigg(i \f{d}{dx}\bigg)_{\max, \min, 0,1,\omega} 
&= T_{\max, \min, 0,1,\omega} (\rho \equiv 1) \, \text{ in } L^2([0,1]; dx),  \lb{3.88c} \\
\bigg(i \f{d}{dx}\bigg)_{\max, \min, 0,1,\omega}^* 
&= T_{\max, \min, 0,1,\omega}^* (\rho \equiv 1) \, \text{ in } L^2([0,1]; dx), 
\end{split}
\end{align} 
taking $\rho\equiv 1$ in \eqref{3.7}--\eqref{3.16}. Introducing the unitary operator 
$U_{\rho}$ via 
\begin{equation}
U_{\rho} \colon \begin{cases} L^2([0,1]; \rho^2 dx) \to  L^2([0,1]; dx) \\
\hspace*{2.05cm}  f \mapsto \rho f, \end{cases}     \lb{3.38d}
\end{equation}
one obtains
\begin{align}
\begin{split}
& U_{\rho} T_{\max, \min, 0,1,\omega} U_{\rho}^{-1} 
= \bigg(i \f{d}{dx}\bigg)_{\max, \min, 0,1,\omega} M_{1/\rho}, \\ 
& U_{\rho} T_{\max, \min, 0,1,\omega}^* U_{\rho}^{-1} 
= M_{1/\rho} \bigg(i \f{d}{dx}\bigg)_{\max, \min, 0,1,\omega}^*,    \lb{3.38e}
\end{split}
\end{align}
where $M_w$ denotes the bounded operator of multiplication by the function $w \in L^\infty([0,1]; dx)$ in the Hilbert space $L^2([0,1]; dx)$, and hence,
\begin{align}
& U_{\rho} T_{\max, \min, 0,1,\omega}^* T_{\max, \min, 0,1,\omega}^{} U_{\rho}^{-1} 
= M_{1/\rho} \bigg(- \f{d^2}{dx^2}\bigg)_{\max, \min, 0,1,\omega} M_{1/\rho}, 
\lb{3.38f} \\ 
& U_{\rho} T_{\max, \min, 0,1,\omega}^{} T_{\max, \min, 0,1,\omega}^* U_{\rho}^{-1} 
= \bigg(i \f{d}{dx}\bigg)_{\max, \min, 0,1,\omega}
M_{1/\rho^2} \bigg(i \f{d}{dx}\bigg)_{\max, \min, 0,1,\omega}^*,    \no
\end{align}
where we abbreviated 
\begin{align}
\bigg(- \f{d^2}{dx^2}\bigg)_{\max, \min, 0,1,\omega}
& = \bigg(i \f{d}{dx}\bigg)_{\max, \min, 0,1,\omega}^*
\bigg(i \f{d}{dx}\bigg)_{\max, \min, 0,1,\omega}   \no \\
& = T_{\max, \min, 0,1,\omega}^* (\rho\equiv 1)  
T_{\max, \min, 0,1,\omega}^{} (\rho\equiv 1)    \lb{3.38g}
\end{align}
in $L^2([0,1]; dx)$, which are obtained as in \eqref{3.29}--\eqref{3.33}, taking 
$\rho \equiv 1$. 

Both relations in \eqref{3.38e} immediately follow from a consideration of 
\begin{equation}
(f, T g)_{L^2([0,1]; \rho^2 dx)}, \quad f \in L^2([0,1]; \rho^2 dx), \;   
g \in \dom (T)   \lb{3.38h}
\end{equation}
respectively, 
\begin{equation}
(f, T^* g)_{L^2([0,1]; \rho^2 dx)}, \quad f \in L^2([0,1]; \rho^2 dx), \; g \in \dom (T^*).  
\lb{3.38i}
\end{equation}

Relation \eqref{3.38f} implies the following trace ideal result:

\begin{theorem} \lb{t3.3}
Assume Hypothesis \ref{h3.1} and let 
\begin{equation}
z \in \rho(T_{\max, \min, 0,1,\omega}^* T_{\max, \min, 0,1,\omega}^{})
\cap \rho(T_{\max, \min, 0,1,\omega}^{} T_{\max, \min, 0,1,\omega}^*).    \lb{3.38j}
\end{equation}
Then 
\begin{align}
\begin{split}
& \big(T_{\max, \min, 0,1,\omega}^* T_{\max, \min, 0,1,\omega}^{} - z I\big)^{-1} 
\in \cB_1\big(L^2([0,1]; \rho^2 dx)\big),   \\
& \big(T_{\max, \min, 0,1,\omega}^{} T_{\max, \min, 0,1,\omega}^* - z I\big)^{-1} 
\in \cB_1\big(L^2([0,1]; \rho^2 dx)\big).   \lb{3.38k}
\end{split}
\end{align}
and 
\begin{align}
\begin{split}
& T_{\max, \min, 0,1,\omega} 
\big(T_{\max, \min, 0,1,\omega}^* T_{\max, \min, 0,1,\omega}^{} - z I\big)^{-1} 
\in \cB_2\big(L^2([0,1]; \rho^2 dx)\big),   \\
& T_{\max, \min, 0,1,\omega}^*
\big(T_{\max, \min, 0,1,\omega}^{} T_{\max, \min, 0,1,\omega}^* - z I\big)^{-1} 
\in \cB_2\big(L^2([0,1]; \rho^2 dx)\big).   \lb{3.38l}
\end{split}
\end{align}
\end{theorem}
\begin{proof}
Since \eqref{3.38k}, in the case of 
$T_{\max, \min, 0,1,\omega}^* T_{\max, \min, 0,1,\omega}^{}$, is well-known if $\rho \equiv 1$ (in which case one can explicitly determine the asymptotics of the $n$th eigenvalue involved 
as $\Oh(n^2)$ as $n\to\infty$), \eqref{3.38k} follows from \eqref{3.38f} and the fact that the nonzero eigenvalues of $(T^* T)^{1/2}$ and $(T T^*)^{1/2}$ 
(i.e., the singular values of $T$ and $T^*$, respectively) coincide, as recalled in 
\eqref{3.38b}. Using the familiar polar decomposition formulas 
\begin{equation}
T = V_T |T|, \quad T^* = (V_T)^* |T^*|, \, \text{ where } \,  |T| = (T^* T)^{1/2}, 
\; |T^*| = (T T^*)^{1/2},  
\lb{3.38m}
\end{equation}
for any closed densely defined operator $T$ in $\cH$, with $V_T$ a partial isometry 
in $\cH$ (cf.\ \cite[Sect.\ VI.2.7]{Ka80}), one obtains 
\begin{align}
\begin{split}
T(T^* T -z I_{\cH})^{-1} &= V_T |T| \big(|T|^2 - z I_{\cH}\big)^{-1},   \\ 
T^*(T T^* -z I_{\cH})^{-1} &= (V_T)^* |T^*| \big(|T^*|^2 - z I_{\cH}\big)^{-1},  \lb{3.38n}
\end{split}
\end{align}
proving \eqref{3.38l}.
\end{proof}

Of course, \eqref{3.38k} and \eqref{3.38l} extend to all $z$ in the corresponding resolvent set of $T^* T$ and $T T^*$, respectively. 

\medskip
 
We conclude this section by listing the Green's functions (i.e., the integral kernels of 
the resolvent of) $T^*T$ at $z=0$ whenever the corresponding nullspaces 
are trivial (cf.\ \eqref{3.18}--\eqref{3.22}). We write\footnote{We denote by $I$ the identity operator in $L^2([0,1]; \rho^2 dx)$) and similarly, by $I_2$ the identity operator in $L^2([0,1]; \rho^2 dx)^2$.}
\begin{align} 
\big((T^* T -z I)^{-1} f\big)(x) = \int_0^1 \rho(x')^2 dx' \, G_{T^* T}(z,x,x') f(x'),&   
\lb{3.39} \\
z \in \rho(T^*T), \; f \in L^2([0,1]; \rho^2 dx),&   \no 
\end{align}
where $T$ represents one of $T_{\min, 0,1,\omega}$, and record the following result:

\begin{lemma} \lb{l3.4}
Assume Hypothesis \ref{h3.1}. Then,
\begin{align}
& G_{T_{\min}^*  T_{\min}^{} } (0,x,x') = \begin{cases} x(1-x'), & 0 \leq x \leq x' \leq 1, \\ 
x'(1-x), & 0 \leq x' \leq x \leq 1, \end{cases}    \lb{3.40} \\ 
& G_{T_0^* T_{0}^{}} (0,x,x') = \begin{cases} x, & 0 \leq x \leq x' \leq 1, \\ 
x', & 0 \leq x' \leq x \leq 1, \end{cases}    \lb{3.41} \\ 
& G_{T_1^* T_{1}^{}} (0,x,x') = \begin{cases} (1-x), & 0 \leq x \leq x' \leq 1, \\ 
(1-x'), & 0 \leq x' \leq x \leq 1, \end{cases}    \lb{3.42} \\
& G_{T_{\omega}^*  T_{\omega}^{} } (0,x,x') = -\f{1}{2} |x-x'| 
- \f{1 + \omega - {\ol \omega} - |\omega|^2}{2|1-\omega|^2} x 
- \f{1 - \omega + {\ol \omega} - |\omega|^2}{2|1-\omega|^2} x'     \no \\
& \hspace*{2.65cm} + \f{1}{|1-\omega|^2}, \quad \omega 
\in \bbC\backslash\{0,1\}.   \lb{3.43}
\end{align}
\end{lemma}
\begin{proof}
In all cases depicted in \eqref{3.40}--\eqref{3.43}, one has 
$(T^* T)^{-1} \in \cB(L^2([0,1]; \rho^2 dx))$ and hence $G_{T^* T}(0,x,x')$ is 
well-defined. Since $T^* T$ is of the type $-(1/\rho^2) (d^2/dx^2)$, one infers 
that for all $c, d \in\bbC$, $T^* T (c + dx) =0$ is valid in the sense of distributions, and hence \eqref{3.40}--\eqref{3.42} readily follow from the separated boundary conditions imposed in \eqref{3.30}, \eqref{3.31}, and \eqref{3.32}. The nonseparated boundary conditions in $T_{\omega}^*  T_{\omega}^{} $ in 
\eqref{3.33} require a slightly different strategy: Introducing  
\begin{align}
g_{\omega} (x; f) &= \int_0^1 \rho(x')^2 dx' \, 
G_{T_{\omega}^*  T_{\omega}^{} }(0,x,x') f(x'), \quad f \in L^2([0,1]; \rho^2 dx), 
\lb{3.44}  \\
g (x; f) &= - \f{1}{2} \int_0^1 \rho(x')^2 dx' \, 
|x-x'| f(x'), \quad f \in L^2([0,1]; \rho^2 dx),    \lb{3.45}
\end{align}
one infers that
\begin{equation}
- \f{1}{\rho (x)^2} \f{d^2}{dx^2} [g_{\omega}(x; f) - g(x; f)] = 0,
\end{equation}
and hence,
\begin{equation}
g_{\omega}(x; f) = g(x; f) + c_1(f) + c_2(f) x, \quad x \in [0,1],
\end{equation}
for some coefficients $c_k(f) \in\bbC$, $k=1,2$.

Since by definition, $g_{\omega}(\cdot; f) \in \dom (T_{\omega}^*  T_{\omega}^{} )$, the boundary conditions (cf.\ \eqref{3.33})
\begin{equation}
g_{\omega} (1; f) = \omega g_{\omega} (0; f), \quad 
g'_{\omega} (1; f) = (1/{\ol \omega}) g'_{\omega} (0; f), 
\end{equation}
and Cramer's rule determine $c_k(f)$, $k=1,2$, implying  
\eqref{3.43}.
\end{proof}

\section{Non-Self-Adjoint Perturbations of Supersymmetric \\ Dirac-Type 
Operators}   \label{s4}

In this section we present an abstract result\footnote{The notation used 
in this abstract part is also employed in Appendix \ref{sA}.}, the computation  
of the resolvent of a non-self-adjoint perturbation of special diagonal type of a supersymmetric Dirac-type operator.

\begin{hypothesis} \lb{h4.1}
Let $\cH_j$, $j=1,2$, be separable complex Hilbert spaces and let 
\begin{equation}
A: \dom(A) \subseteq \cH_1 \to \cH_2    \lb{4.1}
\end{equation}
be a densely defined closed linear operator. In addition, assume that 
\begin{equation}
V \in \cB(\cH_1).     \lb{4.2}
\end{equation}
\end{hypothesis}

Given Hypothesis \ref{h4.1}, we define the self-adjoint supersymmetric Dirac-type operator $Q$ in $\cH_1 \oplus \cH_2$ by 
\begin{equation}
Q = \begin{pmatrix} 0 & A^* \\ A & 0 \end{pmatrix}, \quad 
\dom(Q) = \dom(A) \oplus \dom(A^*),      \lb{4.3}
\end{equation}
introduce the special diagonal operator $W \in \cB(\cH_1 \oplus \cH_2)$ via  
\begin{equation}
W = \begin{pmatrix}  V & 0 \\ 0 & 0 \end{pmatrix},    \lb{4.4}
\end{equation}
and finally consider the perturbed Dirac-type operator 
\begin{align}
\begin{split}
& Q + W = \begin{pmatrix} 0 & A^* \\ A & 0 \end{pmatrix}  
+ \begin{pmatrix} V & 0 \\ 0 & 0 \end{pmatrix},   \lb{4.5} \\
&\dom(Q + W) = \dom(Q) = \dom(A) \oplus \dom(A^*).    
\end{split}
\end{align}

One then computes the following expression for the resolvent of $Q + W$:

\begin{theorem} \lb{t4.2} 
Assume Hypothesis \ref{h4.1} and choose $\zeta\in\bbC$ such that 
\begin{equation}
\zeta^2 \in \rho(A^* A)\cap \rho(A A^*) \, \text{ and } \,  
\big[I + \zeta V (A^* A - \zeta^2 I_{\cH_1})^{-1}\big]^{-1} \in \cB(\cH_1).
\lb{4.6}
\end{equation} 
Then $\zeta \in \rho(Q+W)$ and 
\begin{align}
& (Q + W - \zeta I_{\cH_1 \oplus \cH_2})^{-1}   \no \\
& \quad = \left(\begin{smallmatrix} 
\zeta (A^* A - \zeta^2 I_{\cH_1})^{-1}  \phantom{A^* A - \zeta^2}
& - \zeta (A^* A - \zeta^2 I_{\cH_1})^{-1} [I + \zeta V (A^* A - \zeta^2 I_{\cH_1})^{-1}]^{-1} \\
\times [I + \zeta V (A^* A - \zeta^2 I_{\cH_1})^{-1}]^{-1}
& \times V A^* (A A^* - \zeta^2 I_{\cH_2})^{-1} + A^* (A A^* - \zeta^2 I_{\cH_2})^{-1}  \\[2mm]
A (A^* A - \zeta^2 I_{\cH_1})^{-1}  \phantom{A^* A-\zeta}
& - A (A^* A - \zeta^2 I_{\cH_1})^{-1} [I + \zeta V (A^* A - \zeta^2 I_{\cH_1})^{-1}]^{-1} \\
\times  [I + \zeta V (A^* A - \zeta^2 I_{\cH_1})^{-1}]^{-1} 
& \times V A^* (A A^* - \zeta^2 I_{\cH_2})^{-1} + \zeta (A A^* - \zeta^2 I_{\cH_2})^{-1} 
\end{smallmatrix}\right).    \lb{4.7}
\end{align}
\end{theorem}
\begin{proof}
A direct computation, using \eqref{A.25} reveals
\begin{align}
& (Q + W - \zeta I_{\cH_1 \oplus \cH_2})^{-1} 
= (Q - \zeta I_{\cH_1 \oplus \cH_2})^{-1} 
\big[I_{\cH_1 \oplus \cH_2} + W (Q - \zeta I_{\cH_1 \oplus \cH_2})^{-1}\big]^{-1} 
\no \\
& \quad = \begin{pmatrix} \zeta (A^* A - \zeta^2 I_{\cH_1})^{-1} 
& A^* (A A^* - \zeta^2 I_{\cH_2})^{-1}  \\
A (A^* A - \zeta^2 I_{\cH_1})^{-1} & \zeta (A A^* - \zeta^2 I_{\cH_2})^{-1} \end{pmatrix}   \no \\
& \qquad \times \left(\begin{pmatrix} I_{\cH_1} & 0 \\ 0 & I_{\cH_2} \end{pmatrix}
+ \begin{pmatrix} V & 0 \\ 0 & 0 \end{pmatrix}
\begin{pmatrix} \zeta (A^* A - \zeta^2 I_{\cH_1})^{-1} & A^* (A A^* - \zeta^2 I_{\cH_2})^{-1}  \\ 
A (A^* A - \zeta^2 I_{\cH_1})^{-1} 
& \zeta (A A^* - \zeta^2 I_{\cH_2})^{-1} \end{pmatrix} \right)^{-1} 
\no \\
& \quad = \begin{pmatrix} \zeta (A^* A - \zeta^2 I_{\cH_1})^{-1} 
& A^* (A A^* - \zeta^2 I_{\cH_2})^{-1}  \\
A (A^* A - \zeta^2 I_{\cH_1})^{-1} & \zeta (A A^* - \zeta^2 I_{\cH_2})^{-1} \end{pmatrix}   \no \\
& \qquad \times \left(\begin{pmatrix} I_{\cH_1} & 0 \\ 0 & I_{\cH_2} \end{pmatrix} 
+ \begin{pmatrix} \zeta V (A^* A - \zeta^2 I_{\cH_1})^{-1} & 
V A^* (A A^* - \zeta^2 I_{\cH_2})^{-1} \\ 0 & 0 \end{pmatrix} \right)^{-1}   \no \\
& \quad = \begin{pmatrix} \zeta (A^* A - \zeta^2 I_{\cH_1})^{-1} 
& A^* (A A^* - \zeta^2 I_{\cH_2})^{-1}  \\
A (A^* A - \zeta^2 I_{\cH_1})^{-1} & \zeta (A A^* - \zeta^2 I_{\cH_2})^{-1} \end{pmatrix}   \no \\
& \qquad \times \left(
\begin{smallmatrix} [I + \zeta V (A^* A - \zeta^2 I_{\cH_1})^{-1}]^{-1}  
& \quad - [I + \zeta V (A^* A - \zeta^2 I_{\cH_1})^{-1}]^{-1} V 
A^* (A A^* - \zeta^2 I_{\cH_2})^{-1}  \\[2mm] 0 & I \end{smallmatrix} \right),   \lb{4.8}
\end{align}
using 
\begin{equation}
\begin{pmatrix} F & G \\ 0 & I_{\cH_2} \end{pmatrix}^{-1} 
=\begin{pmatrix} F^{-1} & - F^{-1} G \\ 0 & I_{\cH_2}\end{pmatrix}    \lb{4.9}
\end{equation}
for $F, F^{-1} \in \cB(\cH_1)$, $G \in \cB(\cH_2,\cH_1)$. Relation \eqref{4.8} 
proves \eqref{4.7}. 
\end{proof}

Of course this result extends to more general perturbations $V$: For instance, comparing $|Q|= \left(\begin{smallmatrix} |A| & 0 \\ 0 & |A^*| 
\end{smallmatrix}\right)$ and $W$, it is clear that $V$ being relatively bounded 
with respect to $(A^* A)^{1/2}$ with relative bound strictly less than $1$, would be sufficient. In particular, the term 
\begin{equation}
V A^* (A A^* - \zeta^2 I_{\cH_2})^{-1} =\big [V (A^* A - \zeta^2 I_{\cH_1})^{-1/2}\big] 
\ol{\big[(A^* A - \zeta^2 I_{\cH_1})^{-1/2} A^*\big]} \in \cB(\cH_1)
\end{equation}  
in \eqref{4.7} is then well-defined. More generally, one could invoke form rather 
than operator perturbations $W$ of $|Q|$.

\section{An Infinite Sequence of Trace Formulas for \\ the Damped String Equation}
\label{s5}

The principal aim of this section is the derivation of an infinite sequence of trace 
formulas for the damped string equation.

Throughout the major part of this section the coefficients $\alpha$ and $\rho$ 
in \eqref{1.1} will be assumed to satisfy the following conditions:

\begin{hypothesis}  \lb{h5.1}
Suppose that 
\begin{equation}
\alpha \in L^\infty([0,1]; dx), \, \text{ $\alpha$ real-valued a.e.\ on $(0,1)$},     \lb{5.1}
\end{equation}
and 
\begin{equation}
0 < \rho \in L^\infty([0,1]; dx), \quad 1/\rho \in L^\infty([0,1]; dx).    \lb{5.2}
\end{equation}
\end{hypothesis}

Introducing the self-adjoint Dirac-type operator $D$ in $L^2([0,1]; \rho^2 dx)^2$ 
by
\begin{align}
D = \begin{pmatrix} 0 & T^* \\ T & 0 \end{pmatrix}, \quad 
\dom(D) = \dom(T) \oplus \dom(T^*),   \lb{5.2a}
\end{align}
the concrete non-self-adjoint perturbations $B$ of $D$ then will be of the special diagonal form 
\begin{equation}
B = \begin{pmatrix} - i \alpha/\rho^2 & 0 \\ 0 & 0 \end{pmatrix} 
\in \cB\big(L^2([0,1]; \rho^2 dx)^2\big).     \lb{5.3}
\end{equation}

Given Hypothesis \ref{h5.1}, we introduce the closed (densely defined) 
Dirac-type operator 
\begin{align}
\begin{split}
& D + B = \begin{pmatrix} 0 & T^* \\ T & 0 \end{pmatrix} + 
\begin{pmatrix} - i \alpha/\rho^2 & 0 \\ 0 & 0 \end{pmatrix},   \lb{} \\
&\dom(D + B) = \dom(D) = \dom(T) \oplus \dom(T^*)   \lb{5.4}
\end{split}
\end{align}
in $L^2([0,1]; \rho^2 dx)^2$, where $T$ represents one of $T_{\min}$, $T_0$, 
$T_1$, and $T_{\omega}$, as singled out in Lemma \ref{l3.4}. 

\begin{remark}  \lb{r5.2}
To make the connection with the abstract damped wave equation discussed in Section \ref{s2}, one identifies, $\cH$ and $L^2([0,1]; \rho^2 dx)^2$, $A$ and $T$, $Q$ and $D$, $B$ and $S$, and $R$ and $- i \alpha/\rho^2$, respectively. 
\end{remark}

Our main aim is to derive an infinite sequence of trace formulas associated with $D + B$, but first we note the following result, denoting by 
\begin{equation}
\Im (S)= (S-S^*)/(2i), \quad S\in\cB(\cH)    \lb{5.4a}
\end{equation}
the imaginary part of $S$: 

\begin{theorem} \lb{t5.3}
Assume Hypothesis \ref{h5.1} and choose $\zeta\in\bbC$ such that 
\begin{equation}
\zeta^2 \in \rho(T^* T)\cap \rho(T T^*) \, \text{ and } \,  
\big[I + \zeta B (T^* T - \zeta^2 I)^{-1}\big]^{-1} \in \cB(\cH_1).
\lb{5.5}
\end{equation} 
Then $\zeta \in \rho(D+B)$ and 
\begin{align}
& (D + B - \zeta I_2)^{-1}   \no \\
& = \left(\begin{smallmatrix} 
\zeta (T^* T - \zeta^2 I)^{-1} [I + \zeta B (T^* T - \zeta^2 I)^{-1}]^{-1} 
& - \zeta (T^* T - \zeta^2 I)^{-1} [I + \zeta B (T^* T - \zeta^2 I)^{-1}]^{-1} \\
& \times B T^* (T T^* - \zeta^2 I)^{-1} + T^* (T T^* - \zeta^2 I)^{-1}  \\[2mm]
T (T^* T - \zeta^2 I)^{-1} [I + \zeta B (T^* T - \zeta^2 I)^{-1}]^{-1}  
& - T (T^* T - \zeta^2 I)^{-1} [I + \zeta B (T^* T - \zeta^2 I)^{-1}]^{-1} \\
& \times B T^* (T T^* - \zeta^2 I)^{-1} + \zeta (T T^* - \zeta^2 I)^{-1} 
\end{smallmatrix}\right),    \lb{5.6}
\end{align}
with 
\begin{align}
& (D + B - \zeta I_2)^{-1} \in \cB_2\big(L^2([0,1]; \rho^2 dx)^2\big) 
\backslash  \cB_1\big(L^2([0,1]; \rho^2 dx)^2\big) ,   \lb{5.7} \\
& \Im \big[(D + B - \zeta I_2)^{-1}\big] \in \cB_1\big(L^2([0,1]; \rho^2 dx)^2\big).  
 \lb{5.8}
\end{align}
Here $T$ and $D$ stand for one of $T_{\max, \min, 0,1,\omega}$ and 
$D_{\max, \min, 0,1,\omega}$, respectively. 
\end{theorem}
\begin{proof}
Equation \eqref{5.6} and
$(D + B - \zeta I_2)^{-1} \in \cB_2\big(L^2([0,1]; \rho^2 dx)^2\big)$
are immediate consequences of Theorems \ref{t3.3} and \ref{t4.2}. To show that 
$(D + B - \zeta I_2)^{-1} \notin \cB_1\big(L^2([0,1]; \rho^2 dx)^2\big)$ we first focus on the term 
$T(T^* T -\zeta^2 I)^{-1}$ in the $(1,2)$-entry of \eqref{5.6} and argue as follows: Since the 
eigenvalues of $T^* T$ are known to be asymptotically of the form $c n^2 + \oh(n^2)$ for some 
$c\neq 0$ as $n\to \infty$, one concludes that 
\begin{equation}
|T|(T^* T + I)^{-1} = |T|\big(|T|^2 + I\big)^{-1} \notin \cB_1 \big(L^2([0,1]; \rho^2 dx)^2\big).   \lb{5.8a}
 \end{equation}
Next, we recall the polar decomposition for the densely defined and closed operator $T$, that is, 
$T = V_T |T|$, or, $|T| = (V_T)^* T$, where $|T|=(T^*T)^{1/2}\geq 0$ and $V_T$ is a partial isometry 
(cf.\ \cite[Sect.\ VI.2.7]{Ka80}). Hence, if one argues by contradiction and assumes that 
$T (T^* T + I)^{-1} \in \cB_1 \big(L^2([0,1]; \rho^2 dx)^2\big)$, multiplication of $|T|(T^* T + I)^{-1}$ 
from the left by the partial isometry $(V_T)^*$, and applying the trace ideal property of $\cB_1(\cdot)$, 
would yield the contradiction $|T|(T^* T + I)^{-1} \in \cB_1 \big(L^2([0,1]; \rho^2 dx)^2\big)$. 
Replacing $T$ by $T^*$ this also yields that 
$T^* (T T^* + I)^{-1} \notin \cB_1 \big(L^2([0,1]; \rho^2 dx)^2\big)$ in the $(2,1)$-entry of \eqref{5.6}. 
Since all other terms in \eqref{5.6} are in $\cB_1 \big(L^2([0,1]; \rho^2 dx)^2\big)$, this proves 
\eqref{5.7}. 

Finally, \eqref{5.8} follows from \eqref{5.7} and 
the resolvent identity 
\begin{equation}
 \Im \big[(D + B - \zeta I_2)^{-1}\big] = (D + B - \zeta I_2)^{-1} 
 [\Im(\zeta) - \Im(B)] \big[(D + B - \zeta I_2)^{-1}\big]^*.    \lb{5.9}
\end{equation}
\end{proof}

Next, we turn to the expansion of 
$\tr_{L^2([0,1]; \rho^2 dx)^2}\big(\Im \big[(D + B - \zeta I_2)^{-1}\big]\big)$ 
as $\zeta\to 0$:

\begin{theorem} \lb{t5.4}
Assume Hypothesis \ref{h5.1} and choose $\zeta\in\bbR\backslash\{0\}$ with 
$|\zeta|$ sufficiently small, such that $\zeta \in \rho(D+B)$. Then,   
\begin{align}
& {\tr}_{L^2([0,1]; \rho^2 dx)^2}\big(\Im \big[(D + B - \zeta I_2)^{-1}\big]\big)  \no \\
& \quad = \Im\big[{\tr}_{L^2([0,1]; \rho^2 dx)}\big(\big(2 \zeta + i \big(\alpha/\rho^2\big)\big) 
(T^* T - \zeta^2 I - \zeta i (\alpha/\rho^2))^{-1}\big)\big]     \no \\
& \quad \; = \sum_{m=0}^\infty t_{2m} \, \zeta^{2m},    \lb{5.10} 
\end{align}
where
\begin{align}
t_0 &= {\tr}_{L^2([0,1]; \rho^2 dx)}\big((\alpha/\rho^2) (T^* T)^{-1}\big),   \lb{5.10a} \\
t_2 &= - {\tr}_{L^2([0,1]; \rho^2 dx)}\big(\big(\alpha/\rho^2\big)(T^* T)^{-1} 
\big(\alpha/\rho^2\big) (T^* T)^{-1} \big(\alpha/\rho^2\big) (T^* T)^{-1}\big)   \no \\
& \quad + 3 {\tr}_{L^2([0,1]; \rho^2 dx)}\big(\big(\alpha/\rho^2\big)(T^* T)^{-2}\big),    \lb{5.11}  \\
& \hspace*{-1mm} \text{ etc.}    \no 
\end{align}
Here $T$, $D$, and $t_m$, $m\in\bbN_0$, stand for one of $T_{\min, 0,1,\omega}$,  
$D_{\min, 0,1,\omega}$, and $t_{\min, 0,1,\omega;m}$, $m\in\bbN_0$, respectively, and we only consider the case $\omega\in\bbC\backslash\{0,1\}$ $($cf.\ \eqref{3.22}$)$. 

Explicitly, one obtains for the coefficient $t_0$:
\begin{align}
t_{\min; 0} & = \int_0^1 dx \, x(1-x) \alpha(x),    \lb{5.11a} \\
t_{0;0} & = \int_0^1 dx \; x \alpha(x),    \lb{5.11b}  \\
t_{1;0} & = \int_0^1 dx \; (1-x) \alpha(x),    \lb{5.11c}  \\
t_{\omega;0} & = \f{1}{|1-\omega|^2} \int_0^1 dx \; [(1- |\omega|^2) x +1] \alpha(x),  
\quad \omega \in \bbC\backslash\{0,1\}.  \lb{5.11d} 
\end{align}
\end{theorem}
\begin{proof}
Using \eqref{5.6}, one computes
\begin{align}
& {\tr}_{L^2([0,1]; \rho^2 dx)^2}\big(\Im\big[(D+B-\zeta I_2)^{-1}\big]\big)    \no \\
& \quad = {\tr}_{L^2([0,1]; \rho^2 dx)}\Big(\Im\Big[\zeta(T^* T - \zeta^2 I)^{-1}
\big[I-\zeta i \big(\alpha/\rho^2\big)(T^* T - \zeta^2 I)^{-1}\big]^{-1}\Big]\Big)  \no \\
& \qquad + {\tr}_{L^2([0,1]; \rho^2 dx)}\Big(\Im\Big[T (T^* T - \zeta^2 I)^{-1} 
\big[I-\zeta i (\alpha/\rho^2)(T^* T - \zeta^2 I)^{-1}\big]^{-1}  \no \\
& \qquad \quad \times  i \big(\alpha/\rho^2\big) T^* (T T^* - \zeta^2 I)^{-1}\Big]\Big)    \no \\ 
& \quad = {\tr}_{L^2([0,1]; \rho^2 dx)}\Big(\Im\Big[\zeta 
\big(T^* T - \zeta^2 I - \zeta i \big(\alpha/\rho^2\big)\big)^{-1}\Big]\Big)   \no \\
& \qquad + {\tr}_{L^2([0,1]; \rho^2 dx)}\Big(\Im\Big[T^* T (T^* T - \zeta^2 I)^{-1}    \no \\
& \hspace*{3.3cm} \times \big(T^* T - \zeta^2 I - \zeta i \big(\alpha/\rho^2\big)\big)^{-1} 
i \big(\alpha/\rho^2\big)\Big]\Big)    \no \\
& \quad = {\tr}_{L^2([0,1]; \rho^2 dx)}\Big(\Im\Big[\zeta 
\big(T^* T - \zeta^2 I - \zeta i \big(\alpha/\rho^2\big)\big)^{-1}\Big]\Big)   \no \\
& \qquad 
- {\tr}_{L^2([0,1]; \rho^2 dx)}\Big(\Im\Big[\big(T^* T - \zeta^2 I - \zeta i \big(\alpha/\rho^2\big)\big)^{-1} 
(- i) \big(\alpha/\rho^2\big)\Big]\Big)    \no \\
& \qquad 
+ {\tr}_{L^2([0,1]; \rho^2 dx)}\Big(\Im\Big[\zeta^2 (T^* T - \zeta^2 I)^{-1}    \no \\
& \hspace*{3.3cm}  \times \big(T^* T - \zeta^2 I - \zeta i 
\big(\alpha/\rho^2\big)\big)^{-1} i \big(\alpha/\rho^2\big)\Big]\Big)    \no \\ 
& \quad = {\tr}_{L^2([0,1]; \rho^2 dx)}\Big(\Im\Big[\big(\zeta + i \big(\alpha/\rho^2\big)\big)  
\big(T^* T - \zeta^2 I - \zeta i \big(\alpha/\rho^2\big)\big)^{-1}\Big]\Big)   \no \\
& \qquad 
+ {\tr}_{L^2([0,1]; \rho^2 dx)}\Big(\Im\Big[\zeta \big(T^* T - \zeta^2 I - \zeta i \big(\alpha/\rho^2\big)\big)^{-1}
\no \\
& \hspace*{3.3cm} \times \zeta i \big(\alpha/\rho^2\big) 
(T^* T - \zeta^2 I)^{-1}\Big]\Big)   \no \\
& \quad = {\tr}_{L^2([0,1]; \rho^2 dx)}\Big(\Im\Big[\big(\zeta + i \big(\alpha/\rho^2\big)\big)  
\big(T^* T - \zeta^2 I - \zeta i \big(\alpha/\rho^2\big)\big)^{-1}\Big]\Big)   \no \\
& \qquad 
- {\tr}_{L^2([0,1]; \rho^2 dx)}\Big(\Im\Big[\zeta \big(T^* T - \zeta^2 I - \zeta i \big(\alpha/\rho^2\big)\big)^{-1} 
\no \\
& \qquad \quad \times 
\big(T^* T - \zeta^2 I - \zeta  i \big(\alpha/\rho^2\big) - T^* T + \zeta^2 I\big) 
(T^* T - \zeta^2 I)^{-1}\Big]\Big)   \no \\
& \quad = {\tr}_{L^2([0,1]; \rho^2 dx)}\Big(\Im\Big[\big(\zeta + i \big(\alpha/\rho^2\big)\big)   
\big(T^* T - \zeta^2 I - \zeta i \big(\alpha/\rho^2\big)\big)^{-1}\Big]\Big)   \no \\ 
& \qquad 
- {\tr}_{L^2([0,1]; \rho^2 dx)}\Big(\Im\Big[\zeta \big(T^* T - \zeta^2 I)^{-1} 
- \zeta \big(T^* T - \zeta^2 I - \zeta i \big(\alpha/\rho^2\big)\big)^{-1}\Big]\Big)  \no \\
& \quad = \Im\Big[{\tr}_{L^2([0,1]; \rho^2 dx)}\Big(\big(\zeta + i \big(\alpha/\rho^2\big)\big) 
\big(T^* T - \zeta^2 I - \zeta i \big(\alpha/\rho^2\big)\big)^{-1}\Big)\Big]   \no \\ 
& \qquad 
- \Im\Big[{\tr}_{L^2([0,1]; \rho^2 dx)}\Big(\zeta \big(T^* T - \zeta^2 I)^{-1} 
- \zeta \big(T^* T - \zeta^2 I - \zeta i \big(\alpha/\rho^2\big)\big)^{-1}\Big)\Big]  \no \\
& \quad = \Im\Big[{\tr}_{L^2([0,1]; \rho^2 dx)}\Big(\big(2 \zeta + i \big(\alpha/\rho^2\big)\big)  
\big(T^* T - \zeta^2 I - \zeta i \big(\alpha/\rho^2\big)\big)^{-1}\Big)\Big]   \no \\
& \quad = {\tr}_{L^2([0,1]; \rho^2 dx)}\Big(\Im\Big[\big(2 \zeta + i \big(\alpha/\rho^2\big)\big)  
\big(T^* T - \zeta^2 I - \zeta i \big(\alpha/\rho^2\big)\big)^{-1}\Big]\Big),    
\lb{5.12}
\end{align}
where we repeatedly used cyclicity of the trace (i.e., ${\tr}_{\cH_2}(AB) = {\tr}_{\cH_1}(BA)$ 
for $A\in \cB(\cH_1,\cH_2)$, $B \in \cB(\cH_2,\cH_1)$ with $AB\in \cB_1(\cH_2)$, 
$BA \in \cB_1(\cH_1)$, cf.\ \cite[Corollary\ 3.8]{Si05}),  
\eqref{A.15}, and the fact that ${\tr}_{\cH}(\Im(S)) = {\Im}_{\cH} (\tr(S))$ for $S\in\cB_1(\cH)$. 

The following elementary computation, again employing cyclicity of the trace,   
\begin{align}
& \Im\Big[{\tr}_{L^2([0,1]; \rho^2 dx)}\Big(\big(2 \zeta + i \big(\alpha/\rho^2\big)\big)  
\big(T^* T - \zeta^2 I - \zeta i \big(\alpha/\rho^2\big)\big)^{-1}\Big)\Big]   \no \\
& \quad = {\tr}_{L^2([0,1]; \rho^2 dx)}\Big(\Im\Big[\big(2 \zeta + i \big(\alpha/\rho^2\big)\big)  
\big(T^* T - \zeta^2 I - \zeta i \big(\alpha/\rho^2\big)\big)^{-1}\Big]\Big)   \no \\
& \quad = \f{1}{2i} {\tr}_{L^2([0,1]; \rho^2 dx)} \Big(\big(2\zeta + i\big(\alpha/\rho^2\big)\big) 
\big(T^*T - \zeta^2 I - \zeta i \big(\alpha/\rho^2\big)\big)^{-1}  \no \\
& \hspace*{3.3cm} - \big(T^*T - \zeta^2 I + \zeta i \big(\alpha/\rho^2\big)\big)^{-1} 
(2\zeta I - i \big(\alpha/\rho^2\big)\big)\Big)   \no \\
& \quad = \f{\zeta}{i}  {\tr}_{L^2([0,1]; \rho^2 dx)} 
\Big(\big(T^*T - \zeta^2 I - \zeta i \big(\alpha/\rho^2\big)\big)^{-1} 
- \big(T^*T - \zeta^2 I + \zeta i \big(\alpha/\rho^2\big)\big)^{-1}\Big)    \no \\
& \qquad + \f{1}{2} {\tr}_{L^2([0,1]; \rho^2 dx)}  \Big(\big(\alpha/\rho^2\big)
\big(T^*T - \zeta^2 I - \zeta i \big(\alpha/\rho^2\big)\big)^{-1}    \no \\
& \hspace*{3.5cm}
+ \big(T^*T - \zeta^2 I + \zeta i \big(\alpha/\rho^2\big)\big)^{-1}  \big(\alpha/\rho^2\big)\Big)    
\no \\
& \quad = \f{\zeta}{i}  {\tr}_{L^2([0,1]; \rho^2 dx)} 
\Big(\big(T^*T - \zeta^2 I - \zeta i \big(\alpha/\rho^2\big)\big)^{-1} 
- \big(T^*T - \zeta^2 I + \zeta i \big(\alpha/\rho^2\big)\big)^{-1}\Big)    \no \\ 
& \qquad + \f{1}{2} {\tr}_{L^2([0,1]; \rho^2 dx)}  \Big(\big(\alpha/\rho^2\big)  
 \big(T^*T - \zeta^2 I - \zeta i \big(\alpha/\rho^2\big)\big)^{-1}    \no \\
& \hspace*{3.5cm}
+ \big(\alpha/\rho^2\big) \big(T^*T - \zeta^2 I + \zeta i \big(\alpha/\rho^2\big)\big)^{-1} \Big),   
\lb{5.12a}
\end{align}
then shows that the expression \eqref{5.12} is indeed even with respect 
to $\zeta\in\bbR$, proving \eqref{5.10}. 

Finally, expanding the leading terms in \eqref{5.10} up to order $\Oh\big(\zeta^2\big)$ then 
yields \eqref{5.11}, and the expressions \eqref{5.11a}--\eqref{5.11d}  are a consequence of 
\eqref{3.40}--\eqref{3.43} and \eqref{5.10a}. 
\end{proof}

\begin{remark} \lb{r5.5}
Assume Hypothesis \ref{h5.1}. \\ 
$(i)$ With respect to the location of the eigenvalues $\lambda_j(D+B)$, $j\in J$, we note the standard fact (cf.\ \cite[Problem\ V.4.8]{Ka80}) that self-adjointness of $D$ 
and $B \in \cB\big(L^2([0,1]; \rho^2 dx)^2\big)$ yield for 
$d(z,\sigma(D))>\|B\|_{L^2([0,1]; \rho^2 dx)^2}$ (with $d(z, \Sigma)$ the distance between $z\in\bbC$ and the set $\Sigma \subset \bbC$) that 
\begin{equation}
\big\|B(D-z I_2)^{-1}\big\|_{L^2([0,1]; \rho^2 dx)^2} 
\leq \|B\|_{L^2([0,1]; \rho^2 dx)^2} 
d(z,\sigma(D))^{-1} < 1, 
\end{equation}
and hence by the usual geometric series argument,   
\begin{align}
& \big\|(D+B- z I_2)^{-1}\big\|_{\cB(L^2([0,1]; \rho^2 dx)^2)} 
 \no \\ 
& \quad 
\leq d(z,\sigma(D))^{-1} \sum_{k=0}^{\infty} [\|B\|_{L^2([0,1]; \rho^2 dx)^2} 
d(z,\sigma(D))^{-1}]^k    \no \\ 
& \quad = \big[d(z,\sigma(D)) - \|B\|_{L^2([0,1]; \rho^2 dx)^2}\big]^{-1}.
\end{align}
Thus, the (necessarily discrete) spectrum of $D+B$ is contained in the strip,   
\begin{equation}
\sigma(D+B) \subseteq \big\{z\in\bbC \,\big|\, \Im(z) \in 
\big[- \|B\|_{L^2([0,1]; \rho^2 dx)^2}, \|B\|_{L^2([0,1]; \rho^2 dx)^2}\big]\big\}.  
\lb{5.57}
\end{equation}
$(ii)$ Moreover, since $T^*T$ and $\alpha/\rho^2$ are invariant with respect to the 
operation of complex conjugation in $L^2([0,1]; \rho^2 dx)$ for 
$T=T_{\min,0,1}$, Lemma \ref{l2.5} 
applies to all nonzero eigenvalues of $D+B$ and hence 
\begin{equation}
\lambda_0 \in \sigma_{\rm p} (D+B) \, \text{ if and only if } \, 
- \ol{\lambda_0} \in \sigma_{\rm p} (D+B), \quad \lambda_0 
\in \bbC\backslash\{0\},      \lb{5.58}
\end{equation}
with geometric and algebraic multiplicities preserved. Here $D$ represents 
one of $D_{\min,0,1}$. In the case of $D_{\omega}$, 
$\omega \in \bbC\backslash \{0\}$, one analogously obtains 
\begin{equation}
\lambda_0 \in \sigma_{\rm p} (D_{\omega}+B) \, \text{ if and only if } \, 
- \ol{\lambda_0} \in \sigma_{\rm p} (D_{\ol \omega}+B), \quad \lambda_0 
\in \bbC\backslash\{0\}, \; \omega \in \bbC\backslash \{0\},    \lb{5.58a}
\end{equation}
since 
\begin{equation}
\begin{pmatrix} \gC & 0 \\ 0 & \gC \end{pmatrix} i G_{T_{\omega},B} 
\begin{pmatrix} \gC & 0 \\ 0 & \gC \end{pmatrix} = - i G_{T_{\ol \omega},B}. 
\end{equation}

More precisely (cf.\ Lemma \ref{l5.6} below), the point spectrum of 
$\sigma(D+B)$ and hence that of $i G_{T,\alpha/\rho^2}$, or equivalently, 
that of the associated quadratic pencil 
$L(z) = z^2 I_{\cH} + z i \alpha \rho^{-2} - T^*T$ in leading order   
(independently of the boundary conditions used at $x=0,1$), consists of 
two infinite sequences $\{\lambda_{\pm,j}(D+B)\}_{j \in\bbN} \subset \bbC$, 
such that 
\begin{equation}
\lambda_{\pm,j}(D+B) \underset{j \to \infty}{=} 
\pm  j \pi \bigg[\int_0^1 dx \, \rho(x)\bigg]^{-1} + \oh(|j|).     \lb{5.59}
\end{equation} 
Moreover, higher-order expansions of the type
\begin{align}
\begin{split}
\lambda_{\pm,j}(D+B) \underset{j \to \infty}{=} 
\pm j c_{-1} + \sum_{m=0}^M c_m (\pm j)^{-m} + \oh\big(|j|^{-M}\big), 
\quad M\in\bbN,& \\
c_{-1} = \bigg[\int_0^1 dx \, \rho(x)\bigg]^{-1},&     \lb{5.60}
\end{split}
\end{align}  
have been studied under a variety of additional smoothness assumptions on 
the coefficients $\rho$ and $\alpha$, for instance, in \cite{BF09}, \cite{Ce85}, 
\cite{CFNQ90}, \cite{CZ94}, \cite{Gu85}, \cite{Pi99}, \cite{Sh96a}, \cite{Sh97}, 
\cite{Sh98}, \cite{Sh99}, \cite{SMDB97}, \cite{Ya10}, and the references cited therein. Multiplicity questions of eigenvalues were discussed in \cite{Na00} 
and \cite{Pi99}, and various eigenvalue inequalities were derived in \cite{Na00}. 
\end{remark}

Since the leading-order asymptotics for the non-self-adjoint operator $D+B$ 
is of importance later in this section and we were unable to locate the result 
in the literature under our general assumptions on the coefficients $\alpha$ 
and $\rho$, we now present a proof of \eqref{5.59}:
 
\begin{lemma} \lb{l5.6}
Assume Hypothesis \ref{h5.1}. Then the leading-order asymptotics 
\eqref{5.59} holds. 
Consequently, the identical leading-order eigenvalue asymptotics applies to 
$i G_{T,\alpha/\rho^2}$ and the associated quadratic pencil 
$L(z) = z^2 I + z i \alpha \rho^{-2} - T^*T$, $z\in\bbC$.  
Here $D$ and $T$ represent any of 
$D_{\max, \min, 0,1,\omega}$ and $T_{\max, \min, 0,1,\omega}$, 
$\omega \in \bbC\backslash\{0\}$, respectively. 
\end{lemma}
\begin{proof}
First we note that by the results of Appendix \ref{sA} (cf.\ \eqref{A.3}, 
\eqref{A.10}, \eqref{A.24}, \eqref{A.29}, and \eqref{A.30}) it suffices to focus 
on the nonnegative self-adjoint operator $D^2$ (rather than on the Dirac-type 
operator $D$) and compare it to $(D+B)^2$. Since $D^2 = T^*T \oplus TT^*$, 
and $T^*T$ and $TT^*$ are nonnegative self-adjoint Sturm--Liouville 
operators with appropriate boundary conditions at $x=0,1$ described in 
Section \ref{s3}, the leading-order  
asymptotics of these operators is well-known to be independent of the 
boundary conditions used in this manuscript and given by (see, e.g., 
\cite[Theorem 4.3.1]{Ze05})
\begin{equation}
\lambda_{j}(T^*T), \lambda_{j}(TT^*) \underset{j \to \infty}{=} 
j^2 \pi^2 \bigg[\int_0^1 dx \, \rho(x)\bigg]^{-2} + \oh\big(j^2\big).     \lb{5.60a}
\end{equation} 

Next, one chooses $\mu>0$ sufficiently large such that
\begin{equation}
\Big\|\ol{(D^2 + \mu I_2)^{-1/2} [BD + DB + B^2] 
(D^2 + \mu I_2)^{-1/2}}\Big\|_{\cB(L^2([0,1];\rho^2dx)^2)} < 1, 
\end{equation}
where we recall our simplifying notation $I_2 = I_{L^2([0,1];\rho^2dx)^2}$, and observe 
that 
\begin{align}
&\ol{(D^2 + \mu I_2)^{-1/2} [BD + DB + B^2] 
(D^2 + \mu I_2)^{-1/2}}    \no \\
& \quad =  (D^2 + \mu I_2)^{-1/2} [BD + B^2] 
(D^2 + \mu I_2)^{-1/2}   \no \\
& \qquad + D(D^2 + \mu I_2)^{-1/2} B (D^2 + \mu I_2)^{-1/2}. 
\end{align} 
Then the identity
\begin{align}
& ((D+B)^2 + \mu I_2)^{-1} = (D^2 + \mu I_2)^{-1/2}   \\
& \qquad \times \Big[I_2 + \ol{(D^2 + \mu I_2)^{-1/2} [BD + DB + B^2] 
(D^2 + \mu I_2)^{-1/2}}\Big]^{-1} (D^2 + \mu I_2)^{-1/2}     \no \\
& \quad = (D^2 + \mu I_2)^{-1/2} [I_2 + C] 
(D^2 + \mu I_2)^{-1/2},
\end{align}
where we introduced the Hilbert--Schmidt operator $C$ 
in $L^2([0,1];\rho^2dx)^2$,
\begin{align}
C &= \Big[I_2 + \ol{(D^2 + \mu I_2)^{-1/2} [BD + DB + B^2] 
(D^2 + \mu I_2)^{-1/2}}\Big]^{-1} - I_2    \no \\
&= - \Big[I_2 + \ol{(D^2 + \mu I_2)^{-1/2} [BD + DB + B^2] 
(D^2 + \mu I_2)^{-1/2}}\Big]^{-1}    \\
& \quad \times \ol{(D^2 + \mu I_2)^{-1/2} [BD + DB + B^2] 
(D^2 + \mu I_2)^{-1/2}} \in \cB_2 \big(L^2([0,1];\rho^2dx)^2\big), 
\no 
\end{align}
recalling 
$(D^2 + \mu I_2)^{-1/2} \in \cB_2 \big(L^2([0,1];\rho^2dx)^2\big)$ 
(cf.\ \eqref{5.7}). Since clearly,
\begin{align} 
\begin{split}
&[I_2 + C]^{-1} = \Big[I_2 + \ol{(D^2 + \mu I_2)^{-1/2} [BD + DB + B^2] 
(D^2 + \mu I_2)^{-1/2}}\Big]    \\
& \hspace*{7.15cm} \in \cB\big(L^2([0,1];\rho^2dx)^2\big),   
\end{split}
\end{align}
Theorem V.11.3 in \cite{GK69} applies and hence 
\begin{equation}
\lim_{j\to \infty} \lambda_j ((D+B)^2 + \mu I_2)/
\lambda_j (D^2 + \mu I_2) = 1,
\end{equation}
and thus also $\lim_{j\to \infty} \lambda_j ((D+B)^2)/\lambda_j (D^2) = 1$. 
Together with \eqref{5.60a} and the results in Appendix \ref{sA} relating 
the spectra of $D$ and $T^*T$, $TT^*$,  
mentioned at the beginning of this proof, yield the leading-order asymptotics 
\eqref{5.59}.
\end{proof}

Next, we recall a useful result due to Livsic \cite{Li57} on the trace of the imaginary part of certain non-self-adjoint operators. But first we mention some preparations: Let $F\in\cB_{\infty}(\cH)$, then $F$ has a 
{\it complete system of eigenvectors and generalized eigenvectors in 
$\cH$} if the smallest linear subspace (i.e., the linear span) of all eigenvectors and generalized eigenvectors of $F$ is dense in $\cH$.  

We continue with Schur's lemma:

\begin{lemma} [\cite{GGK90}, Lemma\ II.3.3, \cite{GK69}, 
Theorem\ V.2.1] \lb{l5.5} ${}$ \\
Let $F\in\cB_{\infty}(\cH)$ and denote by $\cE(F)$ the smallest closed 
linear subspace of $\cH$ $($i.e., the closed linear span$)$ containing all eigenvectors and generalized 
eigenvectors of $F$ corresponding to all nonzero eigenvalues of $F$ 
$($i.e., to $\sigma(F) \backslash \{0\}$$)$. Then there exists an orthonormal basis $\{\varphi_j\}_{j\in J}$ of $\cE(F)$ $($with 
$J \subseteq \bbN$ an appropriate index set$)$, such that with 
respect to the basis $\{\varphi_j\}_{j\in J}$, $F$ is an upper-triangular operator satisfying 
\begin{align}
\begin{split} 
& F \varphi_j = \sum_{k=1}^j F_{k,j} \varphi_k, \quad 
F_{k,j} = (\varphi_k, F \varphi_j)_{\cH}, \quad j, k \in J,    \\
& F_{j,j} = \lambda_j(F), \quad  j \in J,      \lb{5.12b}
\end{split} 
\end{align} 
with 
\begin{equation}
\sigma(F) \backslash \{0\} = \{\lambda_j(F)\}_{j \in J}, \quad 
|J| = \sum_{\lambda \in \sigma(F) \backslash \{0\}} m(\lambda,F). 
\lb{5.12c}
\end{equation}
\end{lemma}

Here $|J|$ abbreviates the cardinality of $J$. The orthonormal basis 
$\{\varphi_j\}_{j\in J}$ of $\cE(F)$ is obtained from building Jordan 
blocks associated with chains 
\begin{align}
A \psi_{1,j} &= \lambda_j(F) \psi_{1,j},    \no \\ 
A \psi_{m,j} &= \lambda_j(F) \psi_{m,j}  \, 
\text{ or } \, A \psi_{m,j} = \lambda_j(F) \psi_{m,j} + \psi_{m-1,j},   
\lb{5.12d} \\
& \hspace*{2.45cm} m= 2,\dots,m_a(\lambda_j(F),F), \; j \in J,   \no 
\end{align}
followed by the Gram--Schmidt orthogonalization procedure of the 
system $\psi_{m,j}$, $m=1,\dots,m_a(\lambda_j(F),F)$, $j \in J$.

Next, we briefly consider bounded operators $G\in\cB(\cH)$ with nonnegative imaginary parts, $\Im(G) \geq 0$. Then (as shown, e.g.,  
in \cite[p.\ 136]{GGK90} in the case $\mu_0 = 0$), 
$f \in \ker(G - \mu_0 I_{\cH})$ for some $\mu \in\bbR$ implies 
$(f,(G - \mu_0 I_{\cH}) f)_{\cH}=0$, hence, 
$(f,(G^* - \mu_0 I_{\cH}) f)_{\cH}=0$, and thus, 
$0 \leq (f, \Im(g) f)_{\cH}=0$. 
Consequently, $\Im(G) f = 0$ (cf.\ \eqref{A.8}) and hence 
$f \in \ker(G^* - \mu_0 I_{\cH})$. The symmetry of this argument with respect to $G$ and $G^*$ yields
\begin{equation}
\ker(G - \mu_0 I_{\cH}) = \ker(G^* - \mu_0 I_{\cH}), \quad 
\mu_0 \in \bbR,     \lb{5.12e}
\end{equation}
and hence also
\begin{equation}
\cH = \ker(G - \mu_0 I_{\cH}) \oplus \ol{\ran(G - \mu_0 I_{\cH})}    \lb{5.12f}
\end{equation}
(rather than the standard 
$\cH = \ker(G^* - \mu_0 I_{\cH}) \oplus \ol{\ran(G - \mu_0 I_{\cH})}$). Relation \eqref{5.12f} is interesting as it implies the following fact: 
\begin{align}
\begin{split} 
& \text{If $\ker(G - \mu_0 I_{\cH})\supsetneqq \{0\}$, then $G$ has no generalized (resp., associated)}    \lb{5.12fa} \\ 
& \quad \text{eigenvector corresponding to the eigenvalue 
$\mu_0 \in \bbR$.}  
\end{split} 
\end{align}
Indeed, assuming that for some $f_0 \in \cH$, 
\begin{equation}
(G - \mu_0 I_{\cH})^2 f_0 = 0 \, \text{ but } \, 
(G - \mu_0 I_{\cH}) f_0 \neq 0 
\end{equation}
then yields 
\begin{equation}
(G - \mu_0 I_{\cH}) f_0 \in [\ker(G - \mu_0 I_{\cH}) 
\cap \ran(G - \mu_0 I_{\cH})]
\end{equation}
and hence contradicts \eqref{5.12f}.

The following trace formula \eqref{5.34}, the centerpiece of the next theorem, 
is a well-known result due to Livsic \cite{Li57}:

\begin{theorem} [\cite{GGK90}, Lemma\ VII.8.2, \cite{GK69}, 
Lemma\ II.4.1] \lb{t5.6} ${}$ \\
Assume that $G \in\cB_{\infty}(\cH)$ and $\Im(G) \geq 0$. In addition, 
denote by $\cE(G)$ the smallest closed linear subspace $($i.e., the 
closed linear span$)$ of all eigenvectors and generalized eigenvectors of $G$ corresponding to the nonzero eigenvalues of $G$. \\
$(i)$ Then $\cE(G) \subseteq \ol{\ran(G)}$. Moreover, 
\begin{equation}
 \text{$G$ has a complete system of eigenvectors and generalized 
 eigenvectors}    \lb{5.12g}
\end{equation}
if and only if 
\begin{equation}
\cE(G) = \ol{\ran(G)},    \lb{5.12h}
\end{equation}
which in turn is equivalent to
\begin{equation}
\cE(G)^{\bot} = \ker(G).    \lb{5.12i}
\end{equation}
$(ii)$ Suppose in addition that $\Im(G) \in \cB_1(\cH)$. Then
\begin{equation}
\sum_{j \in J} \Im(\lambda_j(G)) \leq \tr_{\cH} (\Im(G)).    \lb{5.12j}
\end{equation}
Here $\{\lambda_j(G)\}_{j\in J}$, $J\subseteq \bbN$ an appropriate 
index set, denotes the eigenvalues of $G$ ordered with respect 
to decreasing magnitude, 
\begin{equation}
|\lambda_{j+1} (G)| \leq |\lambda_j(G)|, \quad j, j+1 \in J,   \lb{5.12k}
\end{equation} 
counting algebraic multiplicities 
$($with $|J| = \sum_{\lambda \in \sigma(G) \backslash \{0\}} 
m(\lambda,G)$$)$. In addition,  
\begin{align}
& \quad \sum_{j \in J} \Im(\lambda_j(G)) = \tr_{\cH} (\Im(G))  \lb{5.34} \\
& \text{if and only if}    \no \\
& \quad \, \text{$G$ has a complete system of root vectors $($i.e., 
eigenvectors and}    \no \\
& \qquad \, \text{generalized eigenvectors\,$)$.}       \lb{5.35}
\end{align}
\end{theorem}

For convenience of the reader we present the short argument for the 
inequality \eqref{5.12j} and indicate the equivalence of \eqref{5.34} and 
\eqref{5.35}: Choosing an orthonormal Schur basis 
$\{\phi_j\}_{j\in J}$ for $G$ in $\cE(G)$ (cf.\ \eqref{5.12b}, \eqref{5.12c}),
\begin{equation}
G \phi_j = \sum_{k=1}^j G_{k,j} \phi_k, \quad j \in J,    \lb{5.36}
\end{equation}
where
\begin{equation}
G_{j,j} = (\phi_j, G \phi_j)_{\cH} = \lambda_j(G), \quad 
G_{k,j} = (\phi_k, G \phi_j)_{\cH}, \quad j, k \in J,    \lb{5.37} 
\end{equation}
and an orthonormal basis $\{\chi_k\}_{k \in K}$ of $\cE(G)^{\bot}$, with 
$K \subseteq \bbN$ an appropriate index set, one obtains
\begin{align}
{\tr}_{\cH} (\Im(G)) &= \sum_{j\in J} (\phi_j, \Im(G) \phi_j)_{\cH} 
+ \sum_{k \in K} (\chi_k, \Im(G) \chi_k)_{\cH}    \no \\
& = \sum_{j\in J} \Im(\lambda_j(G)) 
+ \sum_{k \in K} (\chi_k, \Im(G) \chi_k)_{\cH},     \lb{5.38}
\end{align} 
using  
\begin{equation} 
(\phi_j, \Im(G) \phi_j)_{\cH} = \Im((\phi_j, G \phi_j)_{\cH}) 
= \Im(\lambda_j(G)), \quad j \in J.    \lb{5.39}
\end{equation} 
Since by hypothesis $\Im(G) \geq 0$, 
\eqref{5.38} proves \eqref{5.12j}. Moreover, \eqref{5.34} holds if and only 
if $\Im(G)|_{\cE(G)^{\bot}} = 0$. The latter can be shown to be equivalent 
to $\cE(G)^{\bot} = \ker(G)$, which in turn is equivalent to 
$\cE(G) = \ol{\ran(G)}$ by \eqref{5.12f}. 

Since Theorem \ref{t5.6}\,$(i)$ requires completeness of the system of eigenvectors and generalized eigenvectors of $G$, we next recall a sufficient criterion for completeness convenient for our subsequent purpose:

\begin{theorem} [\cite{GGK90}, Theorem\ XVII.5.1, \cite{Ma88}, 
Theorem\ I.4.3]  \lb{t5.7}  ${}$ \\
Let $S$ be a self-adjoint operator in $\cH$ with purely discrete spectrum, 
or equivalently, satisfying  $(S- z I_{\cH})^{-1}\in\cB_{\infty} (\cH)$ for some 
$($and hence for all\,$)$ $z \in \rho(S)$. Moreover, let $T$ be 
an $S$-compact operator in $\cH$, that is, 
$T(S- z I_{\cH})^{-1}\in\cB_{\infty}(\cH)$ for some 
$($and hence for all\,$)$ $z \in \rho(S)$. 
In addition, denoting by $\{\lambda_j(S)\}_{j\in J}$, $J\subseteq \bbN$, 
the eigenvalues of $S$ ordered with respect to increasing magnitude, 
$|\lambda_j (S)| \leq |\lambda_{j+1} (S)|$, $j\in J$, counting algebraic multiplicities, assume that for some $p \geq 1$, 
\begin{equation}
\sum_{j\in J'} |\lambda_j(S)|^{-p} < \infty,    \lb{5.40}
\end{equation}
where $J' \subseteq J$ represents the index corresponding to all 
nonzero eigenvalues of $S$. Then $\sigma(S+T)$ consists of only 
discrete eigenvalues in the sense that 
\begin{equation}
(S+T- z I_{\cH})^{-1} \in \cB_{\infty} (\cH), \quad z \in \rho(S+T),   \lb{5.41}
\end{equation} 
and each element $\lambda_k(S+T)$, $k\in K$, $K\subseteq \bbN$, 
of $\sigma(S+T)$ is an isolated point of $\sigma(S+T)$, each 
$\lambda_k(S+T)$ has finite algebraic multiplicity $($i.e., the range 
of the Riesz projection associated with $\lambda_k(S+T)$ 
is finite-dimensional\,$)$, and the system of eigenvectors and generalized eigenvectors of $S+T$ is complete in $\cH$. 
\end{theorem}

In particular, the system of eigenvectors and generalized eigenvectors 
of $(S+T-z I_{\cH})^{-1}$ associated with all $($necessarily nonzero, 
cf.\ \eqref{5.41}$)$ eigenvalues of $(S+T-z I_{\cH})^{-1}$ is complete 
in $\cH$. We also note that the relative compactness assumption on 
$T$ with respect to $S$ yields that 
\begin{equation}
\big\|T (S-z I_{\cH})^{-1}\big\|_{\cB(\cH)} < 1 \, 
\text{ for  $z \in \rho(S)$, $|z|$ sufficiently large,}    \lb{5.42}
\end{equation} 
and hence 
\begin{equation}
(S+T- z I_{\cH})^{-1} = \big[I_{\cH} 
+ T (S - z)^{-1}\big]^{-1} (S - z I_{\cH})^{-1} 
\in \cB_{\infty} (\cH),    \lb{5.42a}
\end{equation} 
for $z \in \rho(S)$, $|z|$ sufficiently large, is well-defined 
(cf.\ \cite[p.\ 200 and Theorem\ 9.7]{We80}), implying \eqref{5.41}. 

\begin{remark} \lb{r5.8}
The proof of Theorem \ref{t5.7} in \cite[Theorem\ XVII.5.1]{GGK90} 
and \cite[Theorem\ I.4.3]{Ma88} relies on a completeness result of 
Keldysh (reproduced in English in \cite[Theorem\ IX.4.1]{GGK90} and 
\cite[Theorem\ V.8.1]{GK69}) 
in the context of certain multiplicative perturbations of compact operators. In 
addition, we note that Theorem \ref{t5.7} extends to a normal operator $S$ 
whose spectrum lies on a finite number of rays starting at $z=0$, moreover, 
it suffices to take $p>0$ (cf.\ \cite[Theorem\ I.4.3]{Ma88}). We also remark that 
Theorem\ XVII.5.1 in \cite{GGK90} requires $S^{-1} \in \cB_{\infty} (\cH)$, 
$T S^{-1} \in \cB_{\infty} (\cH)$, and 
$\big(I_{\cH} + T S^{-1}\big)^{-1} \in \cB(\cH)$, but if $0 \in \sigma(S)$, 
the simple replacement  of $S$ by $S - z_0 I_{\cH}$, for appropriate 
$0 \neq z_0 \in \bbR$,  permits one to remove the restriction of bounded 
invertibility of $S$ (cf.\ \eqref{5.41}). 
\end{remark}

Combining Theorems \ref{t5.3}, \ref{t5.4}, \ref{t5.7} and Theorem \ref{t5.6}\,$(i)$ then yields the following infinite sequence of trace formulas for the damped string equation:

\begin{theorem} \lb{t5.8}
Assume Hypothesis \ref{h5.1} and denote by $\{\lambda_j(D+B)\}_{j\in J}$, $J\subseteq \bbZ$, the eigenvalues of $D+B$ ordered with respect to increasing magnitude, $|\lambda_j (D+B)| \leq |\lambda_{j+1}(D+B)|$, 
$j\in J$, counting algebraic multiplicities. In addition, assume that 
$0 \in \rho(D+B)$. Here $($and below\,$)$ $T$, $D$, and $t_{2n}$, $n\in\bbN_0$, 
stand for one of $T_{\min, 0,1,\omega}$, $D_{\min, 0,1,\omega}$, and 
$t_{\min, 0,1,\omega; 2n}$, $n\in\bbN_0$, respectively, and we only consider the case $\omega\in\bbC\backslash\{0,1\}$ $($cf.\ \eqref{3.22}$)$. 
Then the following infinite sequence of trace formulas hold:
\begin{equation}
\sum_{j\in J} \f{\Im\big(\lambda_j(D+B)^{m+1}\big)}
{|\lambda_j(D+B)|^{2(m+1)}} 
= \begin{cases} -  t_{2 n}, & m=2n, \\ 0, & m=2n+1,  \end{cases}  
\quad n \in \bbN_0.   \lb{5.43}
\end{equation}
Explicitly, one obtains for $m=0,1$ in \eqref{5.43}, 
\begin{align}
& \sum_{j\in J} \f{\Im(\lambda_j(D+B))}{|\lambda_j(D+B)|^2} 
= - {\tr}_{L^2([0,1]; \rho^2 dx)}\big((\alpha/\rho^2) (T^* T)^{-1}\big)   \no \\[1mm]
& \quad = \begin{cases}   - \int_0^1 dx \, x(1-x) \alpha(x),   & T = T_{\min},   \\[1mm]
- \int_0^1 dx \; x \alpha(x),   & T = T_0,   \\[1mm]
- \int_0^1 dx \; (1-x) \alpha(x),   & T = T_1,   \\[1mm]
\f{-1}{|1-\omega|^2} \int_0^1 dx \; [(1- |\omega|^2) x +1] \alpha(x),  
\quad \omega \in \bbC\backslash\{0,1\},   & T = T_{\omega}, 
\end{cases}    \lb{5.44} \\[3mm]
& \sum_{j\in J} \f{\Im(\lambda_j(D+B)) \Re(\lambda_j(D+B))}{|\lambda_j(D+B)|^4} = 0. \lb{5.45} 
\end{align}
\end{theorem}
\begin{proof}
By \eqref{5.7} and \eqref{5.8}, the hypotheses of Theorem \ref{t5.7} are satisfied, identifying $S$ and $P$ with $D$ and $B$, respectively. In particular, $\sigma(D+B)$ consists of only discrete eigenvalues,   $\{\lambda_j(D+B)\}_{j\in J}$, $J\subseteq \bbZ$, and the corresponding system of eigenvectors and generalized eigenvectors of $D+B$ is complete in $L^2([0,1]; \rho^2 dx)^2$. 

However, the crucial hypothesis $\Im(G) \geq 0$ in Theorem \ref{t5.6} is not necessarily implied by Hypothesis \ref{h5.1}. To remedy this fact one can proceed 
as follows: The operators $D+B$ and $D+B+zI_2$, $z \in \bbC$, have the same set of eigenvectors and generalized eigenvectors which are complete in 
$L^2([0,1]; \rho^2 dx)^2$. As in \eqref{5.9} one concludes that for 
$|\zeta| < \varepsilon_0$, $0<\varepsilon_0$ sufficiently small, using that $B$ is purely imaginary, $\Re(B)=0$, cf.\ \eqref{5.3}, 
\begin{align}
&  \Im \big[(D + B - (z+\zeta) I_2)^{-1}\big]   \no \\
& \quad = (D + B - (z+\zeta) I_2)^{-1} 
 [\Im(z+\zeta) - \Im(B)] \big[(D + B - (z + \zeta) I_2)^{-1}\big]^* \geq 0,    \no \\
& \hspace*{5cm} \Im(z) > \varepsilon_0 
+ \|B\|_{\cB(L^2([0,1]; \rho^2 dx)^2)}.    
\lb{5.46}
\end{align}
Indeed, recalling the Hilbert--Schmidt property \eqref{5.7} of the resolvent 
of $D + B$, one infers that \eqref{5.46} is well-defined since 
\begin{align}
& (D + B - (z+\zeta) I_2)^{-1} = (D  - (z+\zeta) I_2)^{-1}    \no \\
& \quad \times \big[I_2 + B 
(D  - (z+\zeta) I_2)^{-1}\big]^{-1} 
\in \cB\big(L^2([0,1]; \rho^2 dx)^2\big),    \no \\
& \hspace*{5.1cm} \Im(z) > \varepsilon_0 + \|B\|_{\cB(L^2([0,1]; \rho^2 dx)^2)}, 
\lb{5.47}
\end{align}
and since 
\begin{align}
\begin{split} 
& \big\| B (D  - (z+\zeta) I_2)^{-1}\big\|_{\cB(L^2([0,1]; \rho^2 dx)^2)} 
 \\
& \quad \leq \| B\|_{\cB(L^2([0,1]; \rho^2 dx)^2)} [\Im(z) - \varepsilon_0]^{-1} 
< 1,    \lb{5.48}
\end{split} 
\end{align}
as $\Im(z) > \varepsilon_0 + \| B\|_{\cB(\cH)}$. 

At this point one can apply Theorem \ref{t5.6}\,$(ii)$ (especially, \eqref{5.34}, 
\eqref{5.35}) to obtain 
\begin{align}
& {\tr}_{L^2([0,1]; \rho^2 dx)^2} \big(\Im\big((D + B - (z+\zeta) I_2)^{-1}\big)\big) 
\no \\
& \quad = \sum_{j\in J} \Im\big((\lambda_j(D + B) - z - \zeta)^{-1}\big) < \infty,    
 \lb{5.49} \\
& \hspace*{.95cm}  |\zeta| < \varepsilon_0, \; 
\Im(z) > \varepsilon_0 + \| B\|_{\cB(\cH)}.    \no 
\end{align} 
Since \eqref{5.49} exhibits no analyticity with respect to $z$ we cannot simply 
continue it to $z=0$. As a result we need to proceed along a different route: 
First we note that $0 \in \rho(D+B)$ and the asymptotic behavior \eqref{5.59} of the 
eigenvalues of $D+B$ yield  
\begin{equation}
\sum_{j \in J} |\lambda_j(D+B)|^{-2} < \infty.     \lb{5.49a}
\end{equation}
Recalling \eqref{5.57}, this also implies 
\begin{equation}
\sum_{j \in J} \big|\Im\big(\lambda_j(D+B)^{-1}\big)\big| = 
\sum_{j\in J} \f{|\Im(\lambda_j(D+B))|}
{|\lambda_j(D+B)|^2} < \infty.      \lb{5.49b}
\end{equation}
Thus, choosing a fixed $\Im(z_1) > \varepsilon_0 + \| B\|_{\cB(\cH)}$, and  
using the Hilbert--Schmidt property \eqref{5.7} of the resolvent 
of $D + B$ once more, one computes with the help of \eqref{5.49} that 
\begin{align}
& {\tr}_{L^2([0,1]; \rho^2 dx)^2} \big(\Im\big((D + B - \zeta I_2)^{-1}\big)\big) 
\no \\
& \quad 
= {\tr}_{L^2([0,1]; \rho^2 dx)^2} \big(\Im\big((D + B - (z_1+\zeta) I_2)^{-1}\big)\big) 
\no \\
& \qquad 
+ {\tr}_{L^2([0,1]; \rho^2 dx)^2} \big(\Im\big((D + B - \zeta I_2)^{-1}\big)\big) 
\no \\
& \qquad 
- {\tr}_{L^2([0,1]; \rho^2 dx)^2} \big(\Im\big((D + B - (z_1 +\zeta) I_2)^{-1}\big)\big) 
\no \\
& \quad 
= {\tr}_{L^2([0,1]; \rho^2 dx)^2} \big(\Im\big((D + B - (z_1 +\zeta) I_2)^{-1}\big)\big) 
\no \\
& \qquad 
+ {\tr}_{L^2([0,1]; \rho^2 dx)^2} \big(\Im\big((D + B - \zeta I_2)^{-1}\big)  
\no \\
& \hspace*{3.45cm} - \Im\big((D + B - (z_1 +\zeta) I_2)^{-1}\big)\big) 
\no \\
& \quad = {\tr}_{L^2([0,1]; \rho^2 dx)^2} 
\big(\Im\big((D + B - (z_1 +\zeta) I_2)^{-1}\big)\big) 
\no \\
& \qquad 
+ \Im\big({\tr}_{L^2([0,1]; \rho^2 dx)^2} \big((D + B - \zeta I_2)^{-1}  
 - (D + B - (z_1 +\zeta) I_2)^{-1}\big)\big) 
\no \\
& \quad = \sum_{j\in J} \Im\big((\lambda_j(D + B) - z_1 - \zeta)^{-1}\big)  \no \\
& \qquad - \Im\bigg(z_1 \sum_{j\in J}  (\lambda_j(D + B) - z_1 - \zeta)^{-1}
(\lambda_j(D + B) - \zeta)^{-1}\bigg)  \no \\
& \quad =  \sum_{j\in J} \Im\big((\lambda_j(D + B) - \zeta)^{-1}\big), \quad 
|\zeta| < \varepsilon_0. 
\end{align}

Thus, one obtains for $\zeta \in (-\varepsilon_0,\varepsilon_0)\backslash\{0\}$, 
$0<\varepsilon_0$ sufficiently small such that 
$(-\varepsilon_0,\varepsilon_0) \subset \rho(D+B)$,  
\begin{align}
& {\tr}_{L^2([0,1]; \rho^2 dx)^2} \big(\Im\big((D + B - \zeta I_2)^{-1}\big)\big) 
 = \sum_{j\in J} \Im\big((\lambda_j(D + B) - \zeta)^{-1}\big)    \no \\
& \quad = \sum_{j\in J} \Im\Big(\lambda_j(D + B)^{-1}
\big(1 - \zeta \lambda_j(D + B)^{-1}\big)^{-1}\Big)    \no \\
& \quad = \sum_{j\in J} \Im\bigg(\sum_{n\in\bbN_0} 
\lambda_j(D + B)^{-1} \big(\zeta \lambda_j(D + B)^{-1}\big)^n\bigg) 
\no \\    
& \quad = \sum_{j\in J} \sum_{n\in\bbN_0} \Im\big(
\lambda_j(D + B)^{-n-1}\big) \zeta^n     \no \\    
& \quad = \sum_{n\in\bbN_0} \bigg(\sum_{j\in J} \Im\big(
\lambda_j(D + B)^{-n-1}\big)\bigg) \zeta^n     \no \\ 
& \quad = - \sum_{n\in\bbN_0} \bigg(\sum_{j\in J} 
\f{\Im\big(\lambda_j(D + B)^{n+1}\big)}
{|\lambda_j (D + B)|^{2(n+1)}}\bigg) \zeta^n     \no \\ 
& \quad = \sum_{n\in\bbN_0} t_{2n} \zeta^{2n},     \lb{5.51}
\end{align}
proving \eqref{5.43} subject to the interchange of the sums over $j$ and $n$. 
To justify this interchange it suffices to prove the absolute convergence of 
$\sum_{j\in J} \sum_{n\in\bbN_0} \Im\big(\lambda_j(D + B)^{-n-1}\big) \zeta^n$. For this purpose one uses
\begin{equation}
\Im\big(z^{n+1}\big) = \Im (z) \sum_{\ell=0}^n z^{\ell} {\ol z}^{n-\ell}, 
\quad z \in \bbC, 
\end{equation}
and estimates 
\begin{equation}
\big|\Im\big(z^{n+1}\big)\big| \leq |\Im(z)| \sum_{\ell=0}^n |z|^{\ell} 
|\ol z|^{n-\ell} = (n+1) |\Im(z)| |z|^n, \quad z \in \bbC.
\end{equation}
By \eqref{5.49b} one estimates, without loss of generality, assuming that 
$|J|=\infty$  (otherwise, there is nothing to prove) and for simplicity, using 
$J=\bbZ\backslash\{0\}$,  
\begin{align}
&\sum_{\substack{j=-N \\ j \neq 0}}^N \sum_{n=0}^N \big|\Im\big(\lambda_j(D + B)^{-n-1}\big) \zeta^n\big|    \no \\
& \quad \leq \sum_{\substack{j=-N \\ j \neq 0}}^N \sum_{n=0}^N (n+1) 
\f{|\Im(\lambda_j(D+B)|}{|\lambda_j(D+B)|^2} |\lambda_j(D+B)|^{-n} 
\big|\zeta\big|^n     \no \\
& \quad \leq C(\varepsilon) + \Bigg( \sum_{j=-N}^{M(\varepsilon)} 
+ \sum_{j=M(\varepsilon)}^N\Bigg) \sum_{n=0}^N 
\f{|\Im(\lambda_j(D+B)|}{|\lambda_j(D+B)|^2} (n+1) 
\big|\zeta/\varepsilon\big|^n \leq C,     \\ 
& \hspace*{9.75cm}   |\zeta| < \varepsilon,   \no 
\end{align}
for some $C(\varepsilon)>0$, $0<M(\varepsilon)<N$, and some 
$C>0$ independent of $N\in\bbN$, proving the interchangeability of the sums 
over $j$ and $n$. Here we used that 
$|\lambda_j(D+B)|\to \infty$ as $|j| \to \infty$ and, since by hypothesis  
$0 \in \rho(D+B)$, $|\lambda_j(D+B)| \geq \varepsilon$, and 
$|\lambda_j(D+B)|^{-n} \leq \varepsilon^{-n}$, $j\in J \cap [M(\varepsilon),\infty)$, for sufficiently small $\varepsilon > 0$.  

Taking $m=0,1$ in \eqref{5.43} and applying \eqref{5.11a}--\eqref{5.11d} then yields \eqref{5.44} 
and \eqref{5.45}. 
\end{proof}

\begin{remark} \lb{r5.11}
$(i)$ Trace formulas, quite different from the ones we derived in 
Theorem \ref{t5.8}, involving certain regularized sums of all eigenvalues were 
derived in \cite{BF09}, \cite{Ce85}, \cite{CFNQ90}, and \cite{Ya10}. For instance,
in the case of Dirichlet boundary conditions at $x=0,1$, and for $\rho(x) \equiv 1$, 
the trace formula 
\begin{align}
\begin{split}
\sum_{j \in \bbN} [\lambda_{-,j}(D_{\rm min}+B) 
+ \lambda_{+,j}(D_{\rm min}+B) - 2 c_0] 
= \f{i}{4} [\alpha(0) + \alpha(1)] + c_0,&  \\
c_0 = - \f{i}{2} \int_0^1 dx \, \alpha (x),&     \lb{5.61}
\end{split}
\end{align} 
was derived in \cite{BF09}. \\
$(ii)$ We emphasize that dissipativity of $B$, more generally, sign-definiteness 
of $\alpha$, is not assumed in this section. 
\end{remark}

Next we briefly turn to the question whether the system of (algebraic) eigenvectors associated with $D+B$ as in Theorem \ref{t5.8} constitutes a Riesz basis. One recalls that a system $\{f_n\}_{n\in\bbN}$ in $\cH$ represents a {\it Riesz basis} in 
$\cH$ if it is equivalent to an orthonormal basis $\{e_n\}_{n\in\bbN}$ in $\cH$, that is, if there exists an operator $V\in\cB(\cH)$ with $V^{-1}\in\cB(\cH)$ and $Vf_n = e_n$, $n\in\bbN$. Equivalently, for any $h \in \cH$, there exists a sequence 
$\{c_n(h)\}_{n\in\bbN} \subset \bbC$, such that 
\begin{equation}
h = \sum_{n\in\bbN} c_n(h) f_n      \lb{5.62}
\end{equation}
converges unconditionally in $\cH$ (i.e., it remains convergent to the same 
element in $\cH$ under any permutation of the terms $c_n(h) f_n$, $n\in\bbN$, 
in \eqref{5.62}). 

More generally, a system $\{g_n\}_{n\in\bbN}$ in $\cH$ represents a 
{\it Riesz basis with parentheses} in $\cH$ if there exists a sequence of strictly increasing positive integers $\{m_{\ell}\}_{\ell\in\bbN} \subset \bbN$, 
$1 \leq m_{\ell} < m_{\ell + 1}$, $\ell\in\bbN$, such that for any 
$h \in \cH$, there exists a sequence $\{d_n(h)\}_{n\in\bbN} \subset \bbC$,
implying the unconditional convergence of the following series over 
$\ell\in\bbN_0$ in $\cH$,  
\begin{equation}
h = \sum_{\ell\in\bbN_0}\Bigg(\sum_{n=m_{\ell} + 1}^{m_{\ell + 1}} 
d_n(h) g_n \Bigg) 
= \sum_{\ell\in\bbN_0} P_{\ell} h,      \lb{5.63}
\end{equation}
where $m_0=0$ and the projections $P_{\ell}$, $\ell \in \bbN_0$, are given by 
\begin{equation}
P_{\ell} h = \sum_{n = m_{\ell} + 1}^{m_{\ell + 1}} d_n(h) g_n, \;\, h \in \cH,  
\quad P_{\ell} P_m = P_m P_{\ell} = \delta_{\ell,m} P_m, \;\, 
\ell, m \in \bbN_0     \lb{5.64}
\end{equation}
(with $m_{\ell}$, $\ell\in\bbN_0$, independent of $h\in\cH$). Equivalently, the 
sequence of subspaces $\{\ran(P_{\ell}\}_{\ell\in\bbN_0}$ with $P_{\ell}$ 
satisfying $P_{\ell} P_m = P_m P_{\ell} = \delta_{\ell,m} P_m$,  
$\ell, m \in \bbN_0$, is called an {\it unconditional $($or Riesz\,$)$ basis of subspaces} in $\cH$.  

The following fundamental abstract result on the existence of a Riesz 
basis with parentheses due to Katsnelson \cite{Ka67a}, \cite{Ka67b}, 
Markus \cite{Ma62}, and Markus and Matsaev \cite{MM81}, \cite{MM84} 
(proved under varying generality) will subsequently be applied to the 
operator $D+B$ in $L^2([0,1]; \rho^2 dx)^2$: 

\begin{theorem} \lb{t5.11}
Let $N$ be a normal operator in $\cH$ with compact resolvent and 
spectrum lying on a finite number of rays 
$\cR_q = \{z\in\bbC\,|\, \arg (z) = \theta_q \in [0, 2\pi)\}$, $q = 1,\dots, r$ 
for some $r \in\bbN$. In addition, suppose that for some $p \in [0,1)$, 
\begin{equation}
\liminf_{t\uparrow \infty} t^{p-1} \#(t,N) < \infty, 
\end{equation}   
where $\#(t,N)$, $t>0$, denotes the sum of the multiplicities of all 
eigenvalues of $N$ in the open disk $D(0;t) \subset \bbC$ of radius 
$t>0$ centered at the origin $z=0$. Assume that $R$ is a densely 
defined operator in $\cH$ satisfying for some $c>0$,
\begin{equation} 
\dom(N) \subseteq \dom(R), \quad  
\|R f\|_{\cB(\cH)} \leq c \|N f\|^p_{\cB(\cH)} \|f\|^{1-p}_{\cH}, 
\;  f \in \dom(N)
\end{equation}
$($i.e., $R$ is $p$-subordinate to $N$, cf.\ \cite[Sect.\ 5]{Ma88} and hence 
$R(N - z I_{\cH})^{-1} \in \cB_{\infty}(\cH)$, $z \in \rho(N)$$)$. 
Then $N+R$ defined on $\dom(N+R) = \dom(N)$ is a densely defined closed 
operator with compact resolvent and the system of root vectors of $N+R$ 
forms a Riesz basis with parentheses in $\cH$. 
\end{theorem}

We note that since $N$ is assumed to be normal, the geometric and 
algebraic multiplicty of all eigenvalues of $N$ coincide in Theorem \ref{t5.11}. 

For a detailed discussion of Theorem \ref{t5.11} 
we refer to Markus \cite[p.\ 27--37, Theorem\ 6.12]{Ma88}. For additional 
results in connection with Theorem \ref{t5.11} we also refer to \cite{Ag76}, 
\cite{Ag77}, \cite{Ag82}, \cite{Ag94}, \cite[Ch.\ 5]{AKSV99}, \cite{Ba83}, 
\cite{Cl68}, \cite[Ch.\ XIX]{DS88a}, \cite{GC01}, \cite{GX04a}, 
\cite[Sect.\ V.4.5]{Ka80}, \cite{Kr57}, \cite{Sa77}, \cite{Sc54}, \cite{Sh79}, 
\cite{Sh90}, \cite{ST08}, \cite{Tu65}, \cite{Wy08}, \cite{Wy10}, \cite{Wy10a}, 
\cite{XG03}, \cite{XY05}, \cite{Zw09}, and the references cited therein. 

Theorem \ref{t2.3} (especially, Theorem \ref{t2.4}) permits one to shift 
spectral considerations of $G_{A,R}$ in $\cH_A \oplus \cH$, or that of 
the quadratic pencil $L(\cdot)$ in \eqref{2.33a}, to that of the simpler  
Dirac-type operator $Q+S$ in $\cH \oplus \cH$. In particular, in the 
concrete context of this section, $B$ is a bounded diagonal perturbation 
of the supersymmetric  self-adjoint Dirac-type operator $D$ in 
$L^2([0,1]; \rho^2 dx)^2$ and an application of Theorem \ref{t5.11} yields 
the following result: 

\begin{theorem} \lb{t5.12}
Assume Hypothesis \ref{h5.1} and introduce the supersymmetric 
Dirac-type operator $D$, the non-self-adjoint diagonal perturbation 
$B$, and the operator $D+B$ in $L^2([0,1]; \rho^2 dx)^2$ by \eqref{5.2a}, 
\eqref{5.3}, and \eqref{5.4}, respectively. Then $D+B$ is closed and densely defined on $\dom(D+B) = \dom(D)$ with compact resolvent. Moreover, the 
system of root vectors of $D+B$, or equivalently, that for $G_{T,\alpha/\rho^2}$ 
$($resp., the quadratic pencil $N(z) = z^2 I + z i B - T^*T$, 
$\dom(N(z) = \dom(T^* T)$, $z \in \bbC$$)$, forms a Riesz basis with 
parentheses in $L^2([0,1]; \rho^2 dx)^2$ $($resp.\ $L^2([0,1]; \rho^2 dx)$$)$. 
More precisely, each of the two sequences 
$\{\lambda_{\pm,j}(D+B)\}_{j\in\bbN}$ in \eqref{5.59} decomposes into finite 
clusters of eigenvalues $\Lambda_{\pm,\ell}$ tending to $\pm\infty$ as 
$\ell\to\pm\infty$ such that the associated Riesz projections 
$P(\Lambda_{\pm,\ell},D+B)$, $\ell \in \bbN_0$, satisfy
\begin{equation}
\sum_{\ell\in\bbN_0} P(\Lambda_{-,\ell},D+B) + 
\sum_{\ell\in\bbN_0} P(\Lambda_{+,\ell},D+B) = I. 
\end{equation}
Here $D$ and $T$ stand for one of $D_{\max, \min, 0,1,\omega}$ and 
$T_{\max, \min, 0,1,\omega}$, respectively.
\end{theorem}
\begin{proof}
That $D+B$ defined on $\dom(D+B) = \dom(D)$ is a densely defined closed operator with compact resolvent is clear since 
$B\in \cB\big(L^2([0,1]; \rho^2 dx)^2\big)$ and the resolvent of $D$ is compact 
in $L^2([0,1]; \rho^2 dx)^2$. In order to apply Theorem \ref{t5.11} we first 
note that by \eqref{5.58} all eigenvalues of $D+B$ lie in a strip symmetric with 
respect to the real axis and that by \eqref{5.58} and \eqref{5.59} the eigenvalues 
of $D+B$ split up into two sequences converging to $\pm\infty$. 
As a result of the proof of Theorem \ref{t5.11} as described in detail in 
\cite[p.\ 27--37]{Ma88}, one can now apply Theorem \ref{t5.11} in the special case 
where $N=D$ is self-adjoint, $r=2$, $R=B$ is bounded, and one chooses 
$p=0$ since $\#(t,D) = \Oh(t)$ as $t\uparrow \infty$, given the asymptotics 
\eqref{5.59}. 
\end{proof}

We conclude with the following observations:

\begin{remark} \lb{r5.13}
It appears that Hypotheses \ref{h5.1} is the most general set of assumptions 
under which the Riesz basis property with parentheses has been proven for 
$D+B$, or equivalently, $G_{T,\alpha/\rho^2}$ (resp., the quadratic 
pencil $N(z) = z^2 I + z i B - T^*T$, $z \in \bbC$). Under additional 
smoothness assumptions on $\alpha$ and $\rho$, and typically for separated 
boundary conditions in $T^*T$, $TT^*$, and/or in the case of damping at 
the end points $x=0,1$, the Riesz basis property (without parentheses) for 
$G_{T,\alpha/\rho^2}$ has been established in \cite{CZ94}, \cite{CZ95}, 
\cite{GP97}, \cite{Sh96a}, \cite{Sh97}, \cite{Sh97a}, \cite{Sh99}, \cite{Sh01}, 
\cite{SMDB97}.   
In this context we recall that the existence of a Riesz basis (without 
parentheses) for $D+B$, or equivalently, $G_{T,\alpha/\rho^2}$, is 
equivalent to both 
operators being unbounded spectral operators in the sense of Dunford 
(cf.\ \cite[Chs.\ XVIII, XIX]{DS88a}). In the case of the boundary conditions 
$T_{\max, \min, 0,1,\omega}$, $\omega \in \bbC \backslash\{\pm 1\}$, one 
can expect the Riesz basis property (without parentheses) to hold under our 
general assumptions in Hypothesis \ref{h5.1}. For a particular class of 
one-dimensional Dirac-type operators the delicate question of unconditional 
convergence of spectral Riesz expansions has recently been treated in great 
detail by Djakov and Mityagin \cite{DM10}, \cite{DM10a},  \cite{DM10b}, 
\cite{Mi04}, with earlier contributions in \cite{CK96}, \cite{HO09}, \cite{OH06}, 
\cite{TY01}, \cite{TY02}. 

We also note that scalar spectral operators were discussed in connection 
with (generalizations of) damped wave equations by Sandefur \cite{Sa77}. 
\end{remark}

\appendix
\section{Supersymmetric Dirac-Type Operators in a Nutshell} \lb{sA}
\renewcommand{\theequation}{A.\arabic{equation}}
\renewcommand{\thetheorem}{A.\arabic{theorem}}
\setcounter{theorem}{0} \setcounter{equation}{0}

In this appendix we briefly summarize some results on supersymmetric 
Dirac-type operators and commutation methods due to \cite{De78}, 
\cite{GSS91}, \cite{Th88}, and \cite[Ch.\ 5]{Th92} (see also \cite{Ha00}). 

The standing assumption in this appendix will be the following:

\begin{hypothesis} \lb{hA.1}
Let $\cH_j$, $j=1,2$, be separable complex Hilbert spaces and 
\begin{equation}
A: \dom(A) \subseteq \cH_1 \to \cH_2    \lb{A.1}
\end{equation}
be a densely defined closed linear operator. 
\end{hypothesis}

We define the self-adjoint Dirac-type operator in $\cH_1 \oplus \cH_2$ by 
\begin{equation}
Q = \begin{pmatrix} 0 & A^* \\ A & 0 \end{pmatrix}, \quad \dom(Q) = \dom(A) \oplus \dom(A^*).     
\lb{A.2}
\end{equation}
Operators of the type $Q$ play a role in supersymmetric quantum mechanics (see, e.g., the extensive list of references in \cite{BGGSS87}). Then,
\begin{equation}
Q^2 = \begin{pmatrix} A^* A & 0 \\ 0 & A A^* \end{pmatrix}     \lb{A.3}
\end{equation}
and for notational purposes we also introduce
\begin{equation}
H_1 = A^* A \, \text{ in } \, \cH_1, \quad H_2 = A A^* \, \text{ in } \, \cH_2.    \lb{A.4}
\end{equation}
In the following, we also need the polar decomposition of $A$ and $A^*$, that is, the representations
\begin{equation} 
A = V_A |A| = |A^*| V_A, \quad A^* = V_{A^*} |A^*| = |A| V_{A^*},     \lb{A.5}
\end{equation}
where
\begin{equation}
|A| = (A^* A)^{1/2} = H_1^{1/2}, \quad |A^*| = (A A^*)^{1/2} = H_2^{1/2}.     \lb{A.6}
\end{equation}
Here $V_A$ is a partial isometry with initial set $\ol{\ran(|A|)}$ and final set $\ol{\ran(A)}$ and 
$V_{A^*}$ is a partial isometry with initial set $\ol{\ran(|A^*|)}$ and final set $\ol{\ran(A^*)}$. In 
particular,
\begin{equation}
V_A = \begin{cases} \ol{A (A^* A)^{-1/2}} = \ol{(A A^*)^{-1/2}A} & \text{on }  (\ker (A))^{\bot},  \\
0 & \text{on }  \ker (A).  \end{cases}     \lb{A.7}
\end{equation}

Next, we collect some properties relating $H_1$ and $H_2$:
\begin{theorem} [\cite{De78}] \lb{tA.2}  
Assume Hypothesis \ref{hA.1} and let $\phi$ be a bounded Borel measurable 
function.  \\
$(i)$ One has
\begin{align}
& \ker(A) = \ker(H_1) = (\ran(A^*))^{\bot}, \quad \ker(A^*) = \ker(H_2) = (\ran(A))^{\bot},  \lb{A.8} \\
& V_A H_1^{n/2} = H_2^{n/2} V_A, \; n\in\bbN, \quad V_A \phi(H_1) = \phi(H_2) V_A.  \lb{A.9} 
\end{align}
$(ii)$ $H_1$ and $H_2$ are essentially isospectral, that is, 
\begin{equation}
\sigma(H_1)\backslash\{0\} = \sigma(H_2)\backslash\{0\},    \lb{A.10}
\end{equation}
in fact, 
\begin{equation}
A^* A [I_{\cH_1} - P_{\ker(A)}] \, \text{ is unitarily equivalent to } \, A A^* [I_{\cH_2} - P_{\ker(A^*)}]. 
\lb{A.10a}
\end{equation} 
In addition,
\begin{align}
& f\in \dom(H_1) \, \text{ and } \, H_1 f = \lambda^2 f, \; \lambda \neq 0,   \no \\
& \quad \text{implies }  \,  A f \in \dom(H_2) \, \text{ and } \, H_2(Af) = \lambda^2 (A f),    \lb{A.11} \\
& g\in \dom(H_2)\, \text{ and } \, H_2 \, g = \mu^2 g, \; \mu \neq 0,     \no \\
& \quad \text{implies }  \, A^* g \in \dom(H_1)\, \text{ and } \, H_1(A^* f) = \mu^2 (A^* g),    \lb{A.12} 
\end{align}
with multiplicities of eigenvalues preserved. \\
$(iii)$ One has for $z \in \rho(H_1) \cap \rho(H_2)$,
\begin{align}
& I_{\cH_2} + z (H_2 - z I_{\cH_2})^{-1} \supseteq A (H_1 - z I_{\cH_1})^{-1} A^*,    \lb{A.13} \\
& I_{\cH_1} + z (H_1 - z I_{\cH_1})^{-1} \supseteq A^* (H_2 - z I_{\cH_2})^{-1} A,    \lb{A.14}
\end{align}
and 
\begin{align}
& A^* \phi(H_2) \supseteq \phi(H_1) A^*, \quad 
A \phi(H_1) \supseteq \phi(H_2) A,   \lb {A.15}   \\
& V_{A^*} \phi(H_2) \supseteq \phi(H_1) V_{A^*}, \quad 
V_A \phi(H_1) \supseteq \phi(H_2) V_A.    \lb {A.15a}
\end{align}
\end{theorem}

As noted by E.\ Nelson (unpublished), Theorem \ref{tA.2} follows from the spectral theorem and the 
elementary identities, 
\begin{align}
& Q = V_Q |Q| = |Q| V_Q,    \lb{A.16} \\
& \ker(Q) = \ker(|Q|) = \ker (Q^2) = (\ran(Q))^{\bot},    \lb{A.17}  \\
& I_{\cH_1 \oplus \cH_2} + z (Q^2 - z I_{\cH_1 \oplus \cH_2})^{-1} 
= Q^2 (Q^2 -z I_{\cH_1 \oplus \cH_2})^{-1} \supseteq Q (Q^2 -z I_{\cH_1 \oplus \cH_2})^{-1} Q,  \no \\
& \hspace*{9.3cm}   z \in \rho(Q^2),    \lb{A.18} \\
& Q \phi(Q^2) \supseteq \phi(Q^2) Q,   \lb{A.19}
\end{align}
where
\begin{equation}
V_Q = \begin{pmatrix} 0 & (V_A)^* \\ V_A & 0 \end{pmatrix}
= \begin{pmatrix} 0 & V_{A^*} \\ V_A & 0 \end{pmatrix}.   \lb{A.20}
\end{equation}

In particular,
\begin{equation}
\ker(Q) = \ker(A) \oplus \ker(A^*), \quad  
P_{\ker(Q)} = \begin{pmatrix} P_{\ker(A)} & 0 \\ 0 & P_{\ker(A^*)} \end{pmatrix},     \lb{A.21}
\end{equation}
and we also recall that
\begin{equation}
\sigma_3 Q \sigma_3 = - Q, \quad \sigma_3 = \begin{pmatrix} I_{\cH_1} & 0 \\ 0 & - I_{\cH_2} 
\end{pmatrix},     \lb{A.22}
\end{equation}
that is, $Q$ and $-Q$ are unitarily equivalent. 

Finally, we note the following relationships between $Q$ and $H_j$, $j=1,2$:

\begin{theorem} [\cite{BGGSS87}, \cite{Th88}] \lb{tA.3}
Assume Hypothesis \ref{hA.1}. \\
$(i)$ Introducing the unitary operator $U$ on $(\ker(Q))^{\bot}$ by
\begin{equation}
U = 2^{-1/2} \begin{pmatrix} I_{\cH_1} & (V_A)^* \\  - V_A & I_{\cH_2} \end{pmatrix} 
\, \text{ on } \,  (\ker(Q))^{\bot},     \lb{A.23}
\end{equation}
one infers that
\begin{equation}
U Q U^{-1} = \begin{pmatrix}  H_1^{1/2} & 0 \\ 0 & - H_2^{1/2} \end{pmatrix} 
\, \text{ on } \,  (\ker(Q))^{\bot}.     \lb{A.24}
\end{equation}
$(ii)$ One has
\begin{align}
\begin{split}
(Q - \zeta I_{\cH_1 \oplus \cH_2})^{-1} = \begin{pmatrix} \zeta (H_1 - \zeta^2 I_{\cH_1})^{-1} 
& A^* (H_2 - \zeta^2 I_{\cH_2})^{-1}  \\  A (H_1 - \zeta^2 I_{\cH_1})^{-1}  & 
\zeta (H_2 - \zeta^2 I_{\cH_2})^{-1}  \end{pmatrix},&    \\
\zeta^2 \in \rho(H_1) \cap \rho(H_2).&   \lb{A.25}
\end{split}
\end{align}
$(iii)$ In addition, 
\begin{align}
\begin{split} 
& \begin{pmatrix} f_1 \\ f_2 \end{pmatrix} \in \dom(Q) \, \text{ and } \, 
Q \begin{pmatrix} f_1 \\ f_2 \end{pmatrix} = \eta \begin{pmatrix} f_1 \\ f_2 \end{pmatrix}, \; \eta \neq 0,  
 \\
& \quad \text{ implies } \, f_j \in \dom (H_j) \, \text{ and } \, H_j f_j = \eta^2 f_j, \; j=1,2.    \lb{A.26}
\end{split} 
\end{align}
Conversely,
\begin{align}
\begin{split} 
& f \in \dom(H_1) \, \text{ and } H_1 f = \lambda^2 f, \; \lambda \neq 0, \\
& \quad \text{implies } \, \begin{pmatrix} f \\ \lambda^{-1} A f \end{pmatrix} \in \dom(Q) \, \text{ and } \, 
Q \begin{pmatrix} f \\ \lambda^{-1} A f \end{pmatrix} 
= \lambda \begin{pmatrix} f \\ \lambda^{-1} A f \end{pmatrix}.    \lb{A.29}
\end{split} 
\end{align}
Similarly,
\begin{align}
\begin{split} 
& g \in \dom(H_2) \, \text{ and } H_2 \, g = \mu^2 g, \; \mu \neq 0, \\
& \quad \text{implies } \, \begin{pmatrix} \mu^{-1} A^* g \\ g \end{pmatrix} \in \dom(Q) \, \text{ and } \, 
Q \begin{pmatrix} \mu^{-1} A^* g \\ g \end{pmatrix} 
= \mu \begin{pmatrix} \mu^{-1} A^* g \\ g \end{pmatrix}.   \lb{A.30}
\end{split} 
\end{align}
\end{theorem}

\medskip

\noindent {\bf Acknowledgments.}
We are indebted to Steve Cox for very stimulating discussions on this topic and for  initiating our interest in this problem. Moreover, we are indebted to 
Sergei Avdonin, Nigel Kalton, Mark Malamud, Alexander Markus, Andrei Shkalikov, 
Gerald Teschl, Vadim Tkachenko, Yuri Tomilov, Carsten Trunk, and Christian Wyss for very valuable correspondence on various topics on non-self-adjoint spectral problems. 

This paper was partly written when taking part in the international research program on Nonlinear Partial Differential Equations at the Centre for Advanced Study (CAS) at the Norwegian Academy of Science and Letters in Oslo during the academic year 2008Ð-09. F.G.\ gratefully acknowledges the great hospitality at CAS during his five-week stay in May--June, 2009.


\end{document}